\documentclass{article}

\usepackage[dvipsnames]{xcolor}
\usepackage[colorlinks=true,
            linktocpage=true,
            citecolor=RubineRed,
            linkcolor=RoyalBlue]{hyperref}
\usepackage[numbers]{natbib}
\usepackage{amsmath}
\usepackage{amssymb}
\usepackage{amsthm}
  \theoremstyle{definition}
    \newtheorem{para}{}[section]
    \newtheorem{definition}[para]{Definition}
    \newtheorem{example}[para]{Example}
    
  \theoremstyle{theorem}
    \newtheorem{lemma}[para]{Lemma}

    \newtheorem{theorem}[para]{Theorem}
    \newtheorem{proposition}[para]{Proposition}

\usepackage{titlesec}
\titleformat{\section}{\large\centering\bfseries}{\thesection}{1em}{}

\usepackage{doi}
\usepackage[a4paper,margin=2.5cm]{geometry}
\usepackage{stmaryrd}
\usepackage{tikz-cd}
  \usetikzlibrary{arrows,arrows.meta}
  \tikzcdset{arrow style=tikz, diagrams={>=stealth'}}
\usepackage{booktabs}
\usepackage{enumerate}
\usepackage{proof}

\usepackage{mathrsfs}
\usepackage[cal=txupr,scr=boondoxupr,scrscaled=1.050]{mathalfa}

\title{Duality for Constructive Modal Logics:\\ from Sahqlvist to Goldblatt-Thomason}
\author{%
  Jim de Groot$^1$, Ian Shillito$^2$ and Ranald Clouston$^3$ \\[1em]
  \normalsize{$^1$ Mathematical Institute, University of Bern, Bern, Switzerland}\\[-.2em]
  \normalsize{\texttt{jim.degroot@unibe.ch}} \\[.2em]
  \normalsize{$^2$ School of Computer Science, University of Birimingham, Birmingham, UK}\\[-.2em]
  \normalsize{\texttt{i.b.p.shillito@bham.ac.uk}} \\[.2em]
  \normalsize{$^3$ School of Computing, The Australian National University, Canberra, Australia}\\[-.2em]
  \normalsize{\texttt{ranald.clouston@anu.edu.au}}
  }
\date{}

\newcommand{\mc}[1]{\mathcal{#1}}

\newcommand{\ms}[1]{\mathscr{#1}}
\newcommand{\mo}[1]{\mathfrak{#1}}  % for models
\newcommand{\alg}[1]{\mathscr{#1}}  % for models
\renewcommand{\log}[1]{\mathsf{#1}}   % for logics
\newcommand{\ff}{\mathfrak}   % for (prime) filters
\renewcommand{\iff}{\quad\text{iff}\quad}
\renewcommand{\phi}{\varphi}
\newcommand{\Prop}{\mathrm{Prop}}
\newcommand{\Formulas}{\mathrm{Form}}

\newcommand{\subsetsim}{\mathrel{\substack{\textstyle\subset\\[-0.2ex]\textstyle\sim}}}
\newcommand{\SEG}{\mathrm{SEG}}
\newcommand{\false}{\mathord{\mathfrak{ff}}}
\newcommand{\ccc}[1]{\overline{#1}}
\newcommand{\cleq}{\eqslantless}
\newcommand{\cR}{\overline{R}}

\DeclareMathOperator{\upp}{up_{\expl}}
\DeclareMathOperator{\pf}{pf}
\newcommand{\deriv}[3]{{#2}\vdash_{\mathsf{#1}}{#3}}

\DeclareMathOperator{\st}{st}
\DeclareMathOperator{\so}{so}
\newcommand{\FOL}{\mathrm{FOL}}
\newcommand{\SOL}{\mathrm{SOL}}

\mathchardef\hyphen="2D
\newcommand{\ISUP}{\operatorname{ISUP}}

\newcommand{\POS}{\operatorname{POS}}
\newcommand{\BOXAT}{\operatorname{BOX\hyphen AT}}

\newcommand{\rto}[1]{to node[#1]{\footnotesize{$R$}}}
\newcommand{\ito}[1]{to node[#1]{\footnotesize{$\leq$}}}
\newcommand{\To}{\Rightarrow}

\newcommand{\Eta}{\bar{\eta}}
\newcommand{\Thet}{\bar{\theta}}

\newcommand{\CK}{\ensuremath{\mathsf{CK}}}
\newcommand{\CKAlg}{\mathsf{CKAlg}}
\newcommand{\CKAlgAx}[1]{\Alg #1}
\newcommand{\CKFrm}{\mathsf{CKFrm}}
\newcommand{\CKDescr}{\mathsf{CKDescr}}
\newcommand{\op}{\mathrm{op}}
\newcommand{\id}{\mathtt{id}}
\DeclareMathOperator{\Alg}{Alg}

\DeclareMathOperator{\seg}{\mathfrak{seg}}

\newcommand{\isup}{\mathsf{isup}}

\newcommand{\myitem}[1]{%
    \renewcommand{\labelenumi}{\textup{(\theenumi)}}
    \renewcommand{\theenumi}{#1}
    \item%
  }

\usepackage{fontawesome5}
\newcommand{\expl}{\text{\scalebox{.8}{\faBomb}}}

\DeclareMathOperator{\cto}{\mathrel{\kern1pt\dot\to\kern1pt}}
\DeclareMathOperator{\cor}{\mathrel{\kern.5pt\dot\lor\kern.5pt}}

\newcommand{\val}{v}
\renewcommand{\int}[1]{I^#1}
\newcommand{\LT}[1]{\mathbb{LT}^{\mathsf{#1}}}
\newcommand{\eqprv}[1]{\llbracket #1 \rrbracket}
\newcommand{\eqcla}[1]{\eqprv{\Formulas}^{#1}}
\newcommand{\pdt}[1]{\models^{\mathcal{#1}}}

\usepackage{wasysym}
\let\oldDiamond\Diamond
\renewcommand{\Diamond}{%
  \mathchoice{\raisebox{-.9pt}{$\displaystyle\oldDiamond$}}
             {\raisebox{-.9pt}{$\oldDiamond$}}
             {\raisebox{-0.5pt}{$\scriptstyle\oldDiamond$}}
             {\raisebox{-0.2pt}{$\scriptscriptstyle\oldDiamond$}}}

\newcommand{\dbox}{\mathord{\Box\kern-1.2ex\cdot\kern.55ex}}
\newcommand{\ddiamond}{\mathord{\Diamond\kern-1.25ex\cdot\kern.6ex}}
\begin{document}

\maketitle

\begin{abstract}
  \noindent
  We carry out a semantic study of the constructive modal logic $\log{CK}$.
  We provide a categorical duality linking the algebraic and birelational semantics of the logic.
  We then use this to prove Sahlqvist style correspondence and completeness
  results, as well as a Goldblatt-Thomason style
  theorem on definability of classes of frames.
\end{abstract}

%%%%%%%%%%%%%%%%%%%%%%%%%%%%%%%%%%%%%%%%%%%%%%%%%%%%%%%%%%%%%%%%%%%%%%%%%%%%%%%%
\section{Introduction}

  The question of how best to define a basic intuitionistic version of modal logic, particularly in
  the presence of the sometimes omitted $\Diamond$ (possibility) operator, has received a
  wide range of differing answers~\cite{Fis84,Wij90,BeldePRit01,Koj12,BalGaoGenOli24}.
  A plausible minimal answer is Constructive $\log{K}$ (\CK)~\cite{BeldePRit01}, which can be
  defined by extending the usual axioms of intuitionsitic propositional logic with two further
  axioms
  \begin{enumerate}
    \setlength{\itemindent}{1em}
    \myitem{$\mathsf{K_{\Box}}$} \label{ax:Kb}
          $\Box(\phi \to \psi) \to (\Box \phi \to \Box \psi)$
    \myitem{$\mathsf{K_{\Diamond}}$} \label{ax:Kd}
          $\Box(\phi \to \psi) \to (\Diamond \phi \to \Diamond \psi)$
  \end{enumerate}
  where $\Box$ is the necessity operator, and by adding Modus Ponens and the usual necessitation
  rule of modal logic (if $\phi$ is a theorem, then so is $\Box\phi$).
  Alternatively, $\CK$ arises proof-theoretically from restricting a standard
  sequent calculus for the classical modal logic $\log{K}$ to single conclusions.

  Most competing notions of basic intuitionistic modal logic can
  be expressed as axiomatic extensions of \CK, but \CK\, is not too minimal to be given
  \emph{birelational} (i.e.~Kripke style) semantics~\cite{MendeP05}.%
  \footnote{A non-example of this is the logic of Bo\v{z}i\'{c} and Do\v{s}en~\cite{BozDos84},
  which is incomparable to \CK\, but has only been given \emph{trirelational} semantics in which
  $\Box$ and $\Diamond$ do not talk about the same relation.}
  It was recently shown by the authors of this paper~\cite{GroShiClo25} how to capture the most
  well known alternative notions of basic intuitionistic modal logic
  not just as axiomatic extensions of \CK,
  but by corresponding conditions on the birelational frames for \CK,
  creating a new semantic understanding of this zoo of competing logics.
  
  The logic $\CK$ and its extensions have applications ranging from
  knowledge representation~\cite{MendeP05}
  to various flavours of constructive epistemic logic~\cite{Wil92,ArtPro16,Pac24}
  to modelling parallel computation~\cite{Wij90,WijNer05} and evaluation~\cite{Pit91}.
  Besides, the $\Diamond$-free fragment of $\CK$,
  which can be axiomatised simply by dropping the \ref{ax:Kd} axiom,
  has been studied extensively.
  By contrast, the $\Diamond$-free fragment
  of the alternative basic intuitionistic modal logic Intuitionistic
  $\log{K}$~\cite{Fis84} is somewhat
  mysterious, with no known finite axiomatisation~\cite{Gre99,DasMar23}.

  In this paper, we continue the semantic study of $\log{CK}$.
  We begin in Section~\ref{Sec:MainConstructions} by introducing the main constructions of
  our algebraic and frame semantics, and in Section~\ref{sec:duality} we derive
  a categorical duality between the algebraic semantics of
  $\log{CK}$ and a suitable notion of \emph{descriptive} \CK-frames.
  We then use this duality to further study $\log{CK}$:
  In Section~\ref{Sec:Sahlqvist}
          we prove Sahlqvist style correspondence and completeness results
          for the logic. The latter theorem relies on the general
          completeness theorem of extensions of $\log{CK}$ with respect to
          a suitable class of descriptive \CK-frames, which follows
          immediately from the duality.
  Finally, in Section~\ref{Sec:GT} we prove a Goldblatt-Thomason style theorem,
          which describes when certain classes of \CK-frames are definable by formulas.
          This time we rely on the duality to transfer Birkhoff's variety theorem
          from algebras to \CK-frames.

\paragraph{Duality}
  In the field of (modal) logic, duality theorems are used to bridge the
  gap between the algebraic and geometric
  (frame-based) views of a logic.
  They can be used to prove completeness results, as well as (analogues of)
  Sahlqvist completeness theorems, bisimilarity-somewhere-else, and
  Goldblatt-Thomason style theorems.

  The study of dualities dates back to
  Stone's representation theorem for Boolean algebras~\cite{Sto36},
  which establishes a categorical duality between Boolean algebras
  and certain topological spaces (now called Stone spaces).
  Subsequent notable results include
  the McKinsey-Tarski representation theorem for closure algebras
  \cite{McKTar44}, the Priestley duality theorem \cite{Pri70} linking distributive
  lattices and spaces now known as Priestley spaces, and Esakia
  duality for Heyting algebras (the algebraic semantics of intuitionistic logic)~\cite{Esa74,Esa19}.

  Dualities for modal logics often build on a duality for the underlying
  propositional logic. The first such result was Goldblatt's duality for modal
  algebras~\cite{Gol76a,Gol76b}, which extends Stone duality.
  This establishes a duality between the algebraic semantics of normal
  modal logic on the one hand, and Kripke frames with extra structure on the
  other. The resulting frames are called \emph{descriptive} Kripke frames,
  and consist of a Kripke frame together with a collection of
  so-called \emph{admissible} subsets of the frame that satisfy certain
  conditions.
  Similar dualities have been derived for monotone modal
  logic~\cite{HanKup04} and neighbourhood logic~\cite{Dos89},
  where in each case the geometric side of the duality is given by
  frames for the logic with a suitable collection of admissible sets.
  
  Mirroring this, various dualities for modal extensions of positive logic
  were derived by
  extending Priestley duality~\cite{CelJan01,Gro21a,GroPat22-pmwi}.
  When working over intuitionistic logic, one often builds on
  Esakia duality~\cite{Pal04b,WolZak99,Gro22-phd}.
  Again, the geometric side of these dualities is given by frames of the logic
  with a collection of admissible subsets.

  But in the case of $\log{CK}$ something interesting happens: the duality
  between \CK-algebras and a descriptive analogue of \CK-frames cannot
  piggy-back on Esakia duality.
  The intuitive reason is that the birelational semantics for $\log{CK}$
  makes crucial use of the fact that the intuitionistic accessibility relation
  is a preorder, while Esakia duality requires a partial order.
  More precisely, the \CK-frame dual to a \CK-algebra is not based on a set
  of prime filters, but on a set of \emph{segments}. These are pairs
  $(\ff{p}, \Gamma)$ consisting of a prime filter $\ff{p}$ together with
  a set of prime filters $\Gamma$. The intuitionistic accessibility relation
  is then given by inclusion of the first argument, and since we can have
  different segments headed by the same prime filter, this gives rise to a
  preorder but not a partial order.
  Hence the duality of this paper is not merely a mathematical tool for further study of the logic,
  but also has interesting features in its own right.%
  \footnote{It would still be possible to piggy-back on Esakia if we used a
            different semantics, for example trirelational frames.
            Since we insist on using the birelational \CK-frames
            we disregard such alternatives.}

\paragraph{Sahlqvist correspondence and completeness}
  After the introduction of the Kripke semantics for classical modal logic,
  it was rapidly noticed that some modal axioms, such as~$\Box p\to p$ and $\Box p\to\Box\Box p$,
  correspond to properties on frames, e.g.~reflexivity and transitivity. 
  These insights led the community to wonder about the existence
  of a general theory of correspondence.
  
  In 1975, Sahlqvist gathered the results from his master's thesis
  into what constitutes the best-to-date attempt at providing such
  a theory~\cite{Sah75}.
  While modal formulas naturally correspond to second-order frame properties,
  Sahlqvist's work syntactically characterises a large class of modal formulas
  which correspond to computable \emph{first-order} frame properties.
  Moreover, the addition of any subset $\Lambda$ of formulas of this class to the
  axioms of the classical modal logic $\log{K}$ yields a logic that is complete
  with respect to the class of frames captured by the frame properties
  corresponding to $\Lambda$. 
  The first result, which is known as the ``Sahlqvist correspondence theorem'',
  was proved independently by Van Benthem~\cite{Ben75},
  while the second, also called ``Sahlqvist completeness theorem'',
  was given a simplified proof via duality by Sambin and Vaccaro~\cite{SamVac89}.
  
  While Sahlqvist-like results have been proved for other classes of
  logics~\cite{CelJan99,Han03,GerNagVen05,ConPal12,ForMor23},
  none currently exists with \CK\, as a basis.
  For example, while results for distributive modal logic~\cite{GerNagVen05,ConPal12}
  can be extended to some intuitionistic modal logics, in that logic
  $\Diamond\bot$ entails $\bot$, and $\Diamond(p\lor q)$ entails $\Diamond p\lor \Diamond q$,
  which are not valid in \CK.
  In this paper we define a class of formulas for which
  we establish correspondence and completeness results. 
  The class of formulas we capture is more restricted
  than the classical one, as we notably forbid the presence of
  diamonds in the antecedent of implications.
  Despite this restriction, many important axioms fall under
  the scope of our definition, including
  the two intuitionistic versions of the $T$ axiom ($\Box p \to p$ and $p \to \Diamond p$),
  the $\Box$ version of the $4$ axiom ($\Box p \to \Box\Box p$),
  and the seriality axiom ($\Box p \to \Diamond p$).

\paragraph{Definability via a Goldblatt-Thomason theorem}
  A prominent question in the study of (modal) logics and their semantics
  is what classes of frames can be defined as the class of frames satisfying
  some set of formulas.
  Such classes are often called \emph{axiomatic} or \emph{modally definable}.
  In the context of classical normal modal logic, a partial answer to this question
  was given by Goldblatt and Thomason~\cite{GolTho74},
  who proved that an elementary class $\ms{K}$ of Kripke frames is axiomatic if and only if
  it reflects ultrafilter extensions and is closed under p-morphic images,
  generated subframes and disjoint unions.
  Instead of assuming $\ms{K}$ to be axiomatic, it suffices to assume
  that it is closed under ultrafilter extensions.
  The proof in~\cite{GolTho74} essentially dualises Birkhoff's variety
  theorem~\cite{Bir35}, using ultrafilter extensions and the duality for
  modal algebras to bridge the gap between the algebraic semantics and
  Kripke frames.
  A model-theoretic proof was provided almost twenty years later by
  Van Benthem~\cite{Ben93}.

  A similar result for intuitionistic logic (without modal operators)
  was proven by Rodenburg \cite{Rod86} (see also \cite{Gol05}),
  where the interpreting structures are \emph{intuitionistic} Kripke frames and models.
  While this uses the same notion of p-morphic images, generated subframes
  and disjoint unions, it replaces ultrafilter extensions with
  so-called prime filter extensions.
  More recently, Goldblatt-Thomason style theorems for many other logics have
  been proven, including for positive normal modal logic~\cite{CelJan99},
  graded modal logic \cite{SanMa10},
  modal extensions of {\L}ukasiewicz finitely-valued logics~\cite{Teh16,BadCaiNog23},
  and modal logics with a universal modality \cite{SanVir19}.
  General approaches to such theorems for coalgebraic and dialgebraic logics
  were given in~\cite{KurRos07,Gro22gt} and~\cite[Section~11]{Gro22-phd}.
  
  In this paper we derive an analogue of the Goldblatt-Thomason theorem
  for $\log{CK}$. 
  Since our duality for $\log{CK}$ is not based on Esakia duality,
  the general approach from~\cite{Gro22gt} is not applicable. 
  But this does not stop us: using the proof idea from~\cite{GolTho74},
  we can still prove the desired analogue.
  This requires two interesting modifications of the original result.
  First, we modify the definition of the disjoint union of a family of frames.
  Since our frames are equipped with an inconsistent world, the disjoint
  union should identify all these worlds. (From a categorical point of view
  nothing changes, because the resulting construction is still given by
  the coproduct in the category of frames and appropriate morphisms.)
  Second, we replace ultrafilter extensions with \emph{segment extensions}.
  These provide the bridge between \CK-frames and their descriptive counterparts,
  and together with the duality theorem allow us to transfer Birkhoff's
  variety theorem to the class of frames.

\

\emph{The third author, Ranald Clouston, would like to express his profound gratitude towards Rob Goldblatt,
first for teaching him mathematical logic, then for supervising his first ever research project.
Rob's brilliant clarity of communication and gentle, supportive manner made him an ideal first mentor.
This paper is dedicated to Rob.}

%%%%%%%%%%%%%%%%%%%%%%%%%%%%%%%%%%%%%%%%%%%%%%%%%%%%%%%%%%%%%%%%%%%%%%%%%%%%%%%%
\section{Algebras and Frames}\label{Sec:MainConstructions}

In this section we introduce our basic logic (Subsection~\ref{subsec:CK}),
its abstract algebraic semantics (Subsection~\ref{subsec:algsem}),
and its birelational frame semantics (Subsection~\ref{subsec:frames}).
We also define the notion of bounded morphism between frames and models,
which allows us to state a proposition regarding the closure of
axiomatically defined classed of frames.

%===============================================================================
\subsection{Constructive $\log{K}$}\label{subsec:CK}

\begin{definition}
Constructive $\log{K}$ (\CK) is the logic whose set of formulas $\Formulas$ is defined by the
grammar 
\begin{equation*}
    \varphi ::= p \mid \bot \mid \varphi \land \varphi \mid \varphi
    \lor \varphi \mid \varphi \rightarrow \varphi \mid \Box\varphi \mid \Diamond\varphi
\end{equation*}
where $p$ is drawn from a set $\Prop$ of propositional variables.
We may define $\neg\phi$ as $\phi\rightarrow\bot$ and $\top$ as $\neg\bot$.

Where $\Gamma$ is a set of formulas and $\phi$ is a formula, we define 
the consequence relation $\deriv{}{\Gamma}{\phi}$ by extending the usual axioms and rules
of intuitionistic propositional logic with axioms \ref{ax:Kb} and \ref{ax:Kd}, and the rule of
necessitation. For a more careful description of consequence, and some of its properties,
see De Groot et al.~\cite[Section III]{GroShiClo25}.

If $\mathsf{Ax}$ is a set of formulas, or abusing notation, a single formula, we write
$\log{CK}\oplus \mathsf{Ax}$ for the logic that extends \CK\, by permitting all substitution
instances of formulas in $\mathsf{Ax}$ to be used as axioms. We additionally note $\Gamma\vdash_{\mathsf{Ax}}\phi$ for the consequence relation of $\log{CK}\oplus \mathsf{Ax}$.
\end{definition}

For interesting examples of such
axiomatic extensions of \CK, see De Groot et al.~\cite{GroShiClo25}, and Section~\ref{subsec:Sahlqvist} of this paper.

%===============================================================================
\subsection{Algebraic Semantics}\label{subsec:algsem}

As the modal logics we scrutinise are intuitionistic,
the algebraic structures they correspond to are \emph{Heyting algebras}
with operators.

\begin{definition}\label{def:CK-alg}
  A \CK-algebra $\alg{A}$ is a Heyting algebra $(A,\top,\bot,\land,\lor,\to)$ with unary operators
  $\Box, \Diamond : A \to A$ satisfying
  \begin{equation*}
    \Box\top = \top, \qquad
    \Box a \wedge \Box b = \Box(a \wedge b), \qquad
    \Diamond a \leq \Diamond(a \vee b), \qquad
    \Box a \wedge \Diamond b \leq \Diamond(a \wedge b)
  \end{equation*}
  We denote the class of \CK-algebras by $\CKAlg$.
  A \emph{homomorphism} $\alg{A} \to \alg{B}$ of \CK-algebras is a function
  $A \to B$ on the underlying sets that preserves all structure. 
\end{definition}

  This definition of \CK-algebras abuses notation in the usual way, by using the same symbols as
  those of the logic.
  We exploit this in our definition of \emph{interpretation}:

\begin{definition}\label{def:algsem}
  Let $\alg{A}$ be a \CK-algebra.
  A \emph{valuation} on $\alg{A}$ is a function $\val:\Prop\to A$.
  Given such a valuation, and a formula $\phi$,
  we recursively define the \emph{interpretation $\int{\val}(\phi)$ of $\phi$ 
  in $\alg{A}$ via $\val$} 
  as follows, where $\star\in\{\land,\lor,\to\}$ and $\heartsuit\in\{\Box,\Diamond\}$: 
  \begin{center}
  \begin{tabular}{l c l @{\hspace{2cm}} l c l}
  $\int{\val}(p)$ & $:=$ & $\val(p)$ & $\int{\val}(\psi \star \chi)$ & $:=$ & $\int{\val}(\psi)\star\int{\val}(\chi)$ \\
  $\int{\val}(\bot)$ & $:=$ & $\bot$ &  $\int{\val}(\heartsuit\psi)$ & $:=$ & $\heartsuit\int{\val}(\psi)$ \\
  \end{tabular}
  \end{center}
\end{definition}

\begin{definition}\label{def:alg_satisfy}
  Given a formula $\phi$, we say that $\alg{A}$ \emph{satisfies} $\phi$, and write
  $\alg{A}\models\phi$, if $\int{\val}{(\phi)}=\top$ for all valuations $\val$.
  If $\Phi \subseteq \Formulas$ then we write $\alg{A} \models \Phi$ if
  $\alg{A} \models \phi$ for all $\phi \in \Phi$.
  We write $ \Alg \Phi$ for $\{ \alg{A} \in \CKAlg \mid \alg{A} \models \Phi \}$, the collection
  of \CK-algebras satisfying $\Phi$.
  We say that a class $\ms{C} \subseteq \CKAlg$ is \emph{axiomatic} if $\ms{C} = \Alg \Phi$
  for some collection $\Phi$ of formulas.  
\end{definition}

\begin{definition}\label{def:ptd}
  Let $\mathcal{C} \subseteq \CKAlg$ of algebras
  (or, by abuse of notation, a single \CK-algebra).
  For a set $\Lambda$ of formulas, write $\Lambda' \subseteq_f \Lambda$ if
  $\Lambda'$ is a finite subset of $\Lambda$.
  The \emph{logic preserving degrees of truth} $\pdt{C}$ is defined by
  \begin{equation}\label{eq:pdt}
    \Lambda\pdt{C}\phi
      \;\; := \;\; \exists \Lambda' \subseteq_f \Lambda .\,
         \forall \alg{A} \in \mathcal{C} . \,
         \forall \val. \,
         \forall a. \;\; 
           \text{if }\;(\forall \psi \in \Lambda'. \;a \leq \int{\val}(\psi))
         \;\text{ then }\;a \leq \int{\val}(\phi)
  \end{equation}
\end{definition}

  Note that The use of a finite $\Lambda'$ builds in the compactness of the logic,
  following Moraschini~\cite[Remark 2.4]{Mor24}.
  If we did not restrict to a finite subset, we would need to work with
  the \emph{canonical extension} of the Lindenbaum Tarski algebra, defined below, 
  which has infinite meets~\cite{GerHar01}.

Our intention is now to show that for any set of axioms $\mathsf{Ax}$,
the logic $\log{CK}\oplus \mathsf{Ax}$ is exactly
the logic preserving degrees of truth $\pdt{\CKAlgAx{Ax}}$.
We establish this fact via a standard argument involving Lindenbaum-Tarski
algebras, which are defined over equivalence classes of formulas.

\begin{definition}
  Let $\mathsf{Ax} \cup \{ \phi, \psi \} \cup \Lambda \subseteq \Formulas$.
  We say that $\phi$ and $\psi$ are \emph{$\mathsf{Ax}$-equivalent} if
  $\deriv{Ax}{}{\phi \leftrightarrow \psi}$.
  The set of formulas $\mathsf{Ax}$-equivalent to
  $\phi$ is denoted by $\eqprv{\phi}^{\mathsf{Ax}}$.
  The class of all sets $\eqprv{\phi}^{\mathsf{Ax}}$, where $\phi$ ranges over
  $\Formulas$, is denoted $\eqcla{\mathsf{Ax}}$.
\end{definition}

  When there is no danger of confusion we drop the superscript $\mathsf{Ax}$
  and simply write $\eqprv{\phi}$ and $\eqcla{}$. In particular, for the 
  remainder of this section we develop our results for a 
  fixed but arbitrary set of axioms $\mathsf{Ax}$.

\begin{definition}[Lindenbaum-Tarski algebra]\label{def:LTalg}
We define $\LT{Ax}=(\eqcla{\mathsf{Ax}},\top,\bot,\land,\lor,\to,\Box,\Diamond)$
to be the \CK-algebra with the operators defined below,
where $\star\in\{\land,\lor,\to\}$ and $\heartsuit \in\{\Box,\Diamond\}$.
\begin{center}
\begin{tabular}{c @{\hspace{1.5cm}}c @{\hspace{1.5cm}} c @{\hspace{1.5cm}} c}
$\top:=\eqprv{\top}$ &
$\bot:=\eqprv{\bot}$ &
$\eqprv{\phi}\star\,\eqprv{\psi}:=\eqprv{\phi\star\psi}$ &
$\heartsuit\eqprv{\phi} := \eqprv{\heartsuit\phi}$ \\
\end{tabular}
\end{center}
\end{definition}

A crucial property of the Lindenbaum-Tarski algebra $\LT{}$ is that
formulas are interpreted as their equivalence class under the \emph{canonical valuation}
of the next lemma.

\begin{lemma}\label{lem:algtrulem}
Let the canonical valuation $\val$ be defined as $\val(p)=\eqprv{p}$.
Then $\int{\val}_{\scriptscriptstyle\LT{}}(\phi)=\eqprv{\phi}$.
\end{lemma}
\begin{proof}
Straightforward induction on $\phi$.
\end{proof}

Leveraging the properties of $\LT{}$, we can show the coincidence of the two logics
under scrutiny.

\begin{theorem}\label{thm:pdt}
The logic $\log{CK}\oplus \mathsf{Ax}$ is the logic preserving degrees of truth over $\CKAlgAx{\mathsf{Ax}}$.
\end{theorem}

\begin{proof}
Suppose $\Lambda\pdt{\CKAlg^{\mathsf{Ax}}}\phi$ and consider
the finite $\Lambda'$ of definition~\eqref{eq:pdt}.
By its finiteness we can create the conjunction of its elements, $\bigwedge\Lambda'$.
Instantiating the definition with $\LT{}$, the canonical valuation, and
$\eqprv{\bigwedge\Lambda'}$, we have that
$$\text{if }(\forall \psi\in\Lambda'.\;\eqprv{\bigwedge\Lambda'}\leq \int{\val}_{\scriptscriptstyle\LT{}}(\psi))\text{ then }\eqprv{\bigwedge\Lambda'}\leq \int{\val}_{\scriptscriptstyle\LT{}}(\phi)$$
By the properties of meet,
$\forall \psi\in\Lambda'.\;\eqprv{\bigwedge\Lambda'}\leq \eqprv{\psi}$,
and so by Lemma~\ref{lem:algtrulem} we have the antecedent of the implication above.
Therefore $\eqprv{\bigwedge\Lambda'}\leq \int{\val}_{\scriptscriptstyle\LT{}{}}(\phi)$ and so
by Lemma~\ref{lem:algtrulem} again, $\eqprv{\bigwedge\Lambda'}\leq \eqprv{\phi}$.
Hence by a standard argument on Heyting algebras,
$\eqprv{\top}\leq \eqprv{(\bigwedge\Lambda')\to\phi}$,
and so $\eqprv{\top}$ and $\eqprv{(\bigwedge\Lambda')\to\phi}$ belong to the
same equivalence class.
This means that $\deriv{}{}{\top\to((\bigwedge\Lambda')\to\phi)}$, so $\deriv{}{}{(\bigwedge\Lambda')\to\phi}$.
By deduction-detachment (see~\cite[Section III]{GroShiClo25})
$\deriv{}{\bigwedge\Lambda'}{\phi}$, so $\deriv{}{\Lambda'}{\phi}$,
hence $\deriv{}{\Lambda}{\phi}$.

The other direction, showing $\Lambda\pdt{\CKAlgAx{Ax}}\phi$ from $\deriv{}{\Lambda}{\phi}$, is a straightforward soundness argument.
\end{proof}

%===============================================================================
\subsection{Frames, Models, and Bounded Morphisms}\label{subsec:frames}

  We recall the definition of a \CK-frame and define disjoint unions,
  generated subframes and bounded morphic images.

\begin{definition}
  A \emph{\CK-frame} is a tuple $\mo{X} = (X, \expl, \leq, R)$ where
  $(X, \leq)$ is a preordered set of \emph{worlds},
  the \emph{exploding} or \emph{inconsistent world} $\expl \in X$ is maximal with respect to
  $\leq$, and $R$ is a binary relation on $X$ such that
  $\expl R x$ if and only if $x = \expl$, for all $x \in X$.

  We write $R[x]$ for $\{ y \in X \mid xRy \}$,
  and $x \sim y$ where $x \leq y \leq x$,
  i.e.~if two worlds are in the same cluster.
\end{definition}

  Some intution may be helpful for this definition. The two binary relations
  $\leq$ and $R$ represent the \emph{intuitionistic} and \emph{modal}
  reachability relations, respectively.
  The exploding world $\expl$ can be understood as the \emph{unique} world at which all
  formulas are satisfied, including $\bot$.
  This allows us to see that $\lnot\Diamond\bot$
  is not a theorem of the basic modal logic \CK, although it is of many of
  its extensions: it is falsified by any world than can reach $\expl$ via the $R$ relation.

  Let us write $\upp(\mo{X})$ for the collection of upsets of $(X, \leq)$
  that contain $\expl$. These form a Heyting algebra, that we will also denote
  by $\upp(\mo{X})$, with top, 
  bottom, meet and join given by
  $W$, $\{ \expl \}$, intersection and union,
  and implication defined via
  \begin{equation*}
      a \To b
      := \{ x \in X \mid (\forall y)(x \leq y \text{ and } y \in a
                                               \text{ imply } y \in b) \}. \\
  \end{equation*}
  If moreover we define $\dbox, \ddiamond : \upp(\mo{X}) \to \upp(\mo{X})$ by
  \begin{align*}
    \dbox a
      &:= \{ x \in X \mid (\forall y)(\forall z)(x \leq y R z
          \text{ implies } z \in a) \} \\
    \ddiamond a
      &:= \{ x \in X \mid (\forall y)(x \leq y
          \text{ implies } (\exists z)(y R z \text{ and } z \in a)) \}
  \end{align*}
  then we obtain a \CK-algebra $\mo{X}^+ := (\upp(\mo{X}), \dbox, \ddiamond)$,
  called the \emph{complex algebra} of $\mo{X}$.

\begin{lemma}
  Let $\mo{X} = (X, \expl, \leq, R)$ be a \CK-frame.
  Then $\mo{X}^+ = (\upp(\mo{X}), \dbox, \ddiamond)$ is a \CK-algebra.
\end{lemma}
\begin{proof}
  It is routine to verify that meets, joins, $\To$, $\dbox$ and $\ddiamond$ are well defined as
  functions on $\upp(\mo{X})$.
  It is similarly routine to check the (in)equations of Definition~\ref{def:CK-alg}.
\end{proof}
  
\begin{definition}\label{def:modsem}
  A \emph{\CK-model} is a pair $\mo{M} = (\mo{X}, V)$ consisting of a \CK-frame
  $\mo{X}$ and a valuation $V : \Prop \to \upp(\mo{X})$ that assigns to each
  proposition letter $p$ an upset of $\mo{X}$ that contains $\expl$.
  The valuation $V$ can be extended to a map $\Formulas \to \upp(\mo{X})$ by:
  \begin{align*}
    V(\bot)
      &= \{ \expl \}
      &  V(\phi \to \psi)
      &= V(\phi) \To V(\psi) \\
    V(\Box\phi)
      &= \dbox V(\phi)
      &  V(\phi \wedge \psi)
      &= V(\phi) \cap V(\psi) \\
    V(\Diamond\phi)
      &= \ddiamond V(\phi)
      &  V(\phi \vee \psi)
      &= V(\phi) \cup V(\psi)
  \end{align*}
  We write $\mo{M}, x \Vdash \phi$ if $x \in V(\phi)$ and
  $\mo{M} \Vdash \phi$ if $V(\phi) = X$.
  Further, we write $\mo{X} \Vdash \phi$ if $\mo{M} \Vdash \phi$
  for every model of the form $\mo{M} = (\mo{X}, V)$.
\end{definition}

\begin{lemma}\label{lem:frm-to-alg}
  Let $\mo{X}$ be a \CK-frame. Then $\mo{X} \Vdash \phi$ if and only if $\mo{X}^+\pdt{}\phi$
  for all $\phi \in \Formulas$.
\end{lemma}
\begin{proof}
Straightforward, as for all valuations $V$, a satisfied formula $\varphi$ has valuation $X$, which
is exactly the definition of $\top$ in the complex algebra as required.
\end{proof}

  Next, we define bounded morphisms between frames and between models.
  These are bounded with respect to both relations in both backwards and forwards directions,
  and are pointed morphisms with respect to the exploding world.

\begin{definition}\label{def:mor}
  A \emph{bounded morphism} between \CK-frames $\mo{X} = (X, \expl, \leq, R)$
  and $\mo{X}' = (X', \expl', \leq', R')$ is a function $f : X \to X'$ such that
  for all $x, y \in X$ and $z' \in X'$:
  \begin{enumerate}\itemsep=0em \itemindent=1.5em
    \myitem{B$_{\expl}$} \label{it:Bt}
            $f(x) = \expl'$ if and only if $x = \expl$;
    \myitem{F$_{\leq}$} \label{it:Fleq}
            if $x \leq y$ then $f(x) \leq' f(y)$;
    \myitem{B$_{\leq}$} \label{it:Bleq}
            if $f(x) \leq' z'$ then there exists a
            $z \in X$ such that $x \leq z$ and $f(z) = z'$;
    \myitem{F$_R$} \label{it:FR}
            if $x R y$ then $f(x) R' f(y)$;
    \myitem{B$_R$} \label{it:BR}
            if $f(x) R' z'$ then there exists a
            $z \in X$ such that $x R z$ and $f(z) = z'$.
  \end{enumerate}
  A \emph{bounded morphism} between models $\mo{M} = (\mo{X}, V)$
  and $\mo{M}' = (\mo{X}', V')$ is a bounded morphism
  $f : \mo{X} \to \mo{X}'$ such that $V = f^{-1} \circ V'$.
\end{definition}

\begin{lemma}\label{lem:bm_props}
  Let $f : \mo{X} \to \mo{X}'$ be a bounded morphism. Then
  \begin{enumerate}
    \item \label{it:mor-cpx-alg}
          $f^{-1} : (\mo{X}')^+ \to \mo{X}^+$ is a homomorphism of \CK-algebras.
    \item If $f$ is an embedding then $f^{-1}$ is surjective.
    \item If $f$ is surjective then $f^{-1}$ is injective.
  \end{enumerate}
\end{lemma}
\begin{proof}
  The first item requires first, that $f^{-1}$ is a function $\upp(\mo{X}')\to\upp(\mo{X})$, i.e.~that
  for each upset of $X'$ containing $\expl$, its $f$-inverse is an upset containing $\expl$.
  This follows from \eqref{it:Bt} and \eqref{it:Fleq}.
  We then must check that $f^{-1}$ exactly preserves all operators.
  Preservation of $\top$,
  meets, and joins is trivial,
  while $\bot= \{ \expl \}$ follows immediately from \eqref{it:Bt}.
  For the other operators we present $\To$ and $\dbox$, with $\ddiamond$
  following similarly.
  
  Let $a', b' \in \upp(\mo{X}')$ and $x \in X$ such that $f(x) \in a' \To b'$.
  We need to show that $x \in f^{-1}(a') \To f^{-1}(b')$.
  To this end, suppose $x \leq y$ and $y \in f^{-1}(a')$.
  Then $f(y) \in a'$ and $f(x) \leq' f(y)$ by~\eqref{it:Fleq}, hence $f(y) \in b'$,
  so that $y \in f^{-1}(b')$ as required.
  Conversely, let $x \in f^{-1}(a') \To f^{-1}(b')$ and 
  suppose $f(x) \leq' y'$ for some $y' \in a'$.
  Then by~\ref{it:Bleq} there exists a $y \in X$ with $x \leq y$ and $f(y) = y'$. 
  This then implies $y \in f^{-1}(a')$, hence $y \in f^{-1}(b')$,
  so that $y' = f(y) \in b'$. This proves $x \in f^{-1}(a' \To b')$.
  
  For the $\dbox$-case,
  let $f(x) \in \dbox a'$ for some $a' \in (\mo{X}')^+$, and suppose that
  $x \leq y\,R\,z$. Then $f(x)\leq' f(y)\,R'\,f(z)$ by \eqref{it:Fleq} and \eqref{it:FR}.
  Hence $z \in \dbox f^{-1}(a')$ as required. The converse follows similarly,
  using~\eqref{it:FR} and~\eqref{it:BR}.
  
  For the second item, if $a \in \upp(\mo{X})$ then since $f$ is bounded,
  $f[a] \subseteq \upp(\mo{X'})$.
  Then $a \in f^{-1}(f[a])$, and since $f$ is an $\leq$-embedding we also
  have $f^{-1}(f[a]) \subseteq a$. Therefore $a = f^{-1}(f[a])$, so $f^{-1}$
  is surjective.
  
  Third, suppose $a', b' \in \upp(\mo{X'})$ and $f^{-1}(a') = f^{-1}(b')$.
  Let $x' \in a'$. By surjectivity of $f$ we can find some $x \in X$ such that
  $f(x) = x'$. Then $f(x) \in a'$, so $x \in f^{-1}(a') = f^{-1}(b')$,
  and hence $x' = f(x) \in b'$. This proves $a' \subseteq b'$.
  We similarly obtain $b' \subseteq a'$, so $a' = b'$.
  This proves injectivity of $f^{-1}$.
\end{proof}

\begin{proposition}\label{prop:mor-pres-truth}
  Let $\mo{X} = (X, \expl, \leq, R)$ and $\mo{X}' = (X', \expl', \leq', R')$
  be two \CK-frames, and $\mo{M} = (\mo{X}, V)$ and $\mo{M}' = (\mo{X}', V')$
  two \CK-models. Suppose $f : \mo{M} \to \mo{M}'$ is a bounded morphism.
  Then for all formulas $\phi$ and all $x \in X$  we have
  \begin{equation*}
    \mo{M}, x \Vdash \phi \iff \mo{M}', f(x) \Vdash \phi.
  \end{equation*}
\end{proposition}
\begin{proof}
  By induction on the structure of $\phi$, with use of the previous lemma.
\end{proof}

\begin{definition}
  For each $i$ in some index set $I$, let
  $\mo{X}_i = (X_i, \expl_i, \leq_i, R_i)$ be a \CK-frame.
  Then the \emph{disjoint union} of the $\mo{X}_i$ is the disjoint union of frames
  modulo an equivalence relation identifying $\expl_i$ for all $i \in I$. More formally,
  \begin{equation*}
    \coprod_{i \in I} \mo{X}_i = (X, \expl, \leq, R)
  \end{equation*}
  where $X = \bigcup \{ (i, x) \mid i \in I, x \in X_i \setminus \expl_i \} \cup \{ \expl \}$,
  and relations $\leq$ and $R$ given by
  \begin{align*}
    (i, x) \leq (j, y) &\iff i = j \text{ and } x \leq_i y 
    &&&    (i, x) R (j, y) &\iff i = j \text{ and } x R_i y \\
    (i, x) \leq \expl &\iff x \leq_i \expl_i 
    &&&    (i, x) R \expl &\iff x R_i \expl_i \\
    \expl \leq \expl
    &&&&    \expl R \expl
  \end{align*}
\end{definition}

  It is easy to verify that $\coprod_{i \in I} \mo{X}_i$ is a \CK-frame.

\begin{lemma}\label{lem:coprod_and_satisfaction}
  Let $\mo{X}_i = (X_i, \expl_i, \leq_i, R_i)$ be a collection of \CK-frames,
  where $i$ ranges over some index set $I$. Then for all $\phi \in \Formulas$,
  \begin{equation*}
    \coprod_{i \in I} \mo{X}_i \Vdash \phi
      \iff \mo{X}_i \Vdash \phi \text{ for all } i \in I.
  \end{equation*}
\end{lemma}
\begin{proof}
  Suppose $\coprod_{i \in I} \mo{X}_i \Vdash \phi$.
  Let $V$ be a valuation for $\mo{X}_j$ for some $j \in I$.
  Then $V$ is also a valuation for $\coprod_{i \in I} \mo{X}_i$.
  Moreover, with this valuation the inclusion function
  \begin{equation*}
    (\mo{X}_j, V) \to \Big( \coprod_{i \in I} \mo{X}_i, V \Big)
  \end{equation*}
  is a bounded morphism between the resulting models.
  It then follows from Lemma~\ref{prop:mor-pres-truth} that
  $(\mo{X}_j, V) \Vdash \phi$, and since $V$ was arbitrary $\mo{X}_j \Vdash \phi$.
  
  For the converse, suppose $\mo{X}_i \Vdash \phi$ for all $i \in I$.
  Let $V$ be a valuation for $\coprod_{i \in I} \mo{X}_i$.
  For each $j \in I$, define the valuation $V_j$ for $\mo{X}_j$ by
  $V_j(p) = V(p) \cap X_j$. Then for each $j \in I$ the inclusion function
  \begin{equation*}
    (\mo{X}_j, V_j) \to \Big( \coprod_{i \in I} \mo{X}_i, V \Big)
  \end{equation*}
  is a bounded morphism between models.
  Since every world in the coproduct lies in the image of one of the inclusion
  functions, Proposition~\ref{prop:mor-pres-truth} entails
  $(\coprod_{i \in I} \mo{X}_i, V) \Vdash \phi$,
  and hence $\coprod_{i \in I} \mo{X}_i \Vdash \phi$ because the valuation $V$ was arbitrary.
\end{proof}

\begin{lemma}\label{lem:prod-coprod}
  Let $\mo{X}_i = (X_i, \expl_i, \leq_i, R_i)$ be a collection of \CK-frames,
  where $i$ ranges over some index set $I$.
  Then we have
  \begin{equation*}
    \Big( \coprod_{i \in I} \mo{X}_i \Big)^+ \cong \prod_{i \in I} \mo{X}_i^+.
  \end{equation*}
\end{lemma}
\begin{proof}
  The right-to-left function maps each tuple to the disjoint union of its components, while the
  left-to-right function has as $i$'th component the elements of the disjoint union whose first
  component is $i$. These maps are easily seen to be inverses.
\end{proof}

\begin{definition}
  A \emph{generated subframe} of a \CK-frame $\mo{X} = (X, \expl, \leq, R)$ is a \CK-frame
  $(X', \expl, \leq', R')$ where $X'$ is a subset of $X$ containing $\expl$ that is closed upwards for
  both $\leq$ and $R$, and $\leq'$ and $R'$ are the restrictions of $\leq$ and $R$ to $X'$.

  The \CK-frame $\mo{X'}$ is a \emph{bounded morphic image} of $\mo{X}$ if there exists a surjective bounded
  morphism $\mo{X}\to\mo{X'}$
\end{definition}

\begin{lemma}
  Let $\mo{X} = (X, \expl, \leq, R)$ and $\mo{X}' = (X', \expl', \leq', R')$ be
  two \CK-frames such that $X \subseteq X'$.
  Then $\mo{X}$ is a generated subframe of $\mo{X}'$ if and only if
  the inclusion $i : X \to X'$ is an embedding.
\end{lemma}

\begin{lemma}\label{lem:morphisms_and_satisfaction}
  Let $f : \mo{X} \to \mo{X}'$ be a bounded morphism.
  \begin{enumerate}
    \item If $f$ is an embedding, then $\mo{X}' \Vdash \phi$ implies
          $\mo{X} \Vdash \phi$, for all $\phi \in\Formulas$.
    \item If $f$ is surjective, then $\mo{X} \Vdash \phi$ implies
          $\mo{X}' \Vdash \phi$, for all $\phi \in \Formulas$.
  \end{enumerate}
\end{lemma}
\begin{proof}
  Suppose $\mo{X}' \Vdash \phi$ and let $V$ be a valuation for $\mo{X}$.
  Define a valuation $V^{\uparrow}$ for $\mo{X}'$ by
  $V^{\uparrow}(p) = \{ x' \in X' \mid \exists y \in V(p) \text{ s.t. } f(y) \leq x' \}$.
  Then the fact that $f$ is an embedding implies that $V = f^{-1} \circ V^{\uparrow}$.
  Using this, and the fact that $f$ is a bounded morphisms between the \CK-frames
  $\mo{X}$ and $\mo{X}'$, it follows that
  $f : (\mo{X}, V) \to (\mo{X}', V^{\uparrow})$ is a
  bounded morphism between \CK-models.
  The assumption that $\mo{X} \Vdash \phi$, together with
  Proposition~\ref{prop:mor-pres-truth}, implies that $(\mo{X}, V) \Vdash \phi$.
  Since $V$ is arbitrary, we conclude $\mo{X} \Vdash \phi$.

  For the second item, suppose $\mo{X} \Vdash \phi$ and let $V'$ be any
  valuation for $\mo{X}'$. Define the valuation $V$ for $\mo{X}$ by
  $V(p) = f^{-1}(V'(p))$. Then the same argument as above yields $\mo{X}' \Vdash \phi$.
\end{proof}

\begin{definition}\label{def:frm-axiomatic}
  A collection $\mathcal C$ of \CK-frames is called \emph{axiomatic} if there exists a set of
  formulas $\mathsf{Ax}$ such that $\mathcal C$ is the class of all frames that satisfy the theorems
  of $\CK\oplus\mathsf{Ax}$.
\end{definition}

\begin{proposition}[c.f. Definition~\ref{def:alg_satisfy}]\label{prop:ax-class-1}
Axiomatic classes are closed under taking disjoint unions, generated subframes and bounded morphic images.
\end{proposition}
\begin{proof}
By Lemmas~\ref{lem:coprod_and_satisfaction} and~\ref{lem:morphisms_and_satisfaction}.
\end{proof}

%%%%%%%%%%%%%%%%%%%%%%%%%%%%%%%%%%%%%%%%%%%%%%%%%%%%%%%%%%%%%%%%%%%%%%%%%%%%%%%%
\section{Duality}\label{sec:duality}

  Complex algebras provide a method of turning \CK-frames into
  \CK-algebras. Can we also go in the reverse direction?
  In other words, given a \CK-algebra, can we construct a \CK-frame?
  Usually in intuitionistic (modal) logic, a frame constructed from an algebra
  is based on the collection of \emph{prime filters} of the
  algebra, see for example~\cite[Section~8.2]{ChaZak97}, \cite{WolZak99} or~\cite{Pal04b}.
  When ordered by inclusion, these form an intuitionistic Kripke frame.
  Inspired by the classical modal setting~\cite[Section~5.3]{BRV01},
  the modal relation can then be defined by letting $\ff{p} R \ff{q}$ if
  for all elements $a$ of the algebra we have that
  $\Box a \in \ff{p}$ implies $a \in \ff{q}$, and
  $a \in \ff{q}$ implies $\Diamond a \in \ff{p}$.

  However, in case of \CK\ this does not work because the frames constructed
  in this way validate formulas that are not derivable, such as
  $\Diamond(\phi \vee \psi) \to \Diamond \phi \vee \Diamond \psi$.
  To remedy this, the dual of a \CK-algebra is based on \emph{segments},
  i.e.~pairs $(\ff{p}, \Gamma)$ consisting of a prime filter $\ff{p}$
  together with a set of prime filters $\Gamma$ which intuitively encodes
  the successors of the segments.

  In Section~\ref{subsec:generalframes}, we start by defining general \CK-frames.
  These are \CK-frames with a collection of ``admissible'' subsets.
  The idea behind these is that denotations of formulas always need to be admissible.
  We then use segments to construct a general \CK-frame from a \CK-algebra,
  and show that the action of starting with a \CK-algebra, constructing its
  dual general \CK-frame, and then taking the \CK-algebra of admissible subsets
  gives rise to an isomorphism of \CK-algebras.
  
  Next, in Section~\ref{subsec:descriptive} we restrict the class of general
  \CK-frames to \emph{descriptive} \CK-frames in order to obtain a duality between
  \CK-algebras and descriptive \CK-frames. Finally, in Section~\ref{subsec:dual}
  we extend this to morphisms, resulting in a full duality between the category
  of \CK-algebras and homomorphisms on the one hand, and descriptive \CK-frames
  and general bounded morphisms on the other.

%===============================================================================
\subsection{General \CK-Frames}\label{subsec:generalframes}

  We generalise the notion of \CK-frame, by equipping it with a collection of
  admissible subsets. We then evaluate formulas only in admissible sets,
  rather than all upsets containing $\expl$.
  Abstractly, the collection of admissible subsets can be seen as a subalgebra
  of the complex algebra of the frame.

\begin{definition}
  A \emph{general \CK-frame} $\mo{G}$ is a tuple $(X, \expl, \leq, R, A)$ consisting of
  a \CK-frame $(X, \expl, \leq, R)$ and a set $A \subseteq \upp(\mo{X})$
  containing $\{ \expl \}$ and $X$, which is closed under $\cap, \cup$ and
  under the operations $\To, \dbox, \ddiamond$ as defined in Section~\ref{subsec:frames}.
\end{definition}

  It is straightforward to observe that a (complex) algebra may be defined on the set $A$ of
  admissible upsets of a general \CK-frame $\mo{G}$, as the definition ensures closure under
  all operations. We write this \CK-algebra as $\mo{G}^*$.
  
  We now show that every \CK-algebra $\alg{A}$ gives rise to a
  general \CK-frame, denoted by $\alg{A}_*$.

\begin{definition}
  Let $\alg{A}$ be a Heyting algebra.
  An \emph{ideal} is a nonempty downset of $A$ that is closed under finite joins.
  Dually, a \emph{filter} is a nonempty upset $\ff{p}$ of $A$ that is closed under finite
  meets. A filter is \emph{prime} if $a \vee b \in \ff{p}$ implies
  $a \in \ff{p}$ or $b \in \ff{p}$. A (prime) filter is \emph{proper} if it is not the whole of $A$, but
  we do not in general require our filters be proper.
  We write $\pf(A)$ for the set of prime filters of $A$.
  If $a \in A$ then we denote by $\theta(a)$ the set of prime filters containing $a$,
  that is, $\theta(a) := \{ \ff{p} \in \pf(A) \mid a \in \ff{p} \}$.
\end{definition}

\begin{lemma}[Prime Filter Lemma]\label{lem:prime}
  Let $\alg{A}$ be a Heyting algebra, $F$ be a filter of $\alg{A}$ and $I$ be
  an ideal of $\alg{A}$. If $F \cap I = \emptyset$ then there exists a
  prime filter $\ff{p}$ extending $F$ that is disjoint from $I$.
\end{lemma}
\begin{proof}
See e.g.~Wolter and Zakharyaschev~\cite[Theorem 7.41]{WolZak97}.
\end{proof}
  
\begin{definition}
  Let $\alg{A}$ be a \CK-algebra and  $\Gamma \cup \{ \ff{p} \} \subseteq \pf(A)$.
  We call the pair $(\ff{p}, \Gamma)$ a \emph{segment} if for all $a \in A$:
  \begin{itemize}
    \item If $\Box a \in \ff{p}$ then $a \in \ff{q}$ for all $\ff{q} \in \Gamma$;
    \item If $\Diamond a \in \ff{p}$ then $a \in \ff{q}$ for some $\ff{q} \in \Gamma$.
  \end{itemize}
  Let $\SEG$ be the set of all segments.
  Let $\expl$ be $(A,\{A\})$,
  and define binary relations $\subsetsim$ and $R_A$ on $\SEG$ by
  \begin{align*}
    (\ff{p}, \Gamma) \subsetsim (\ff{q}, \Delta) &\iff \ff{p} \subseteq \ff{q} \\
    (\ff{p}, \Gamma) R_A (\ff{q}, \Delta) &\iff \ff{q} \in \Gamma
  \end{align*}
  Then we denote by $\alg{A}_+$ the \CK-frame
  \begin{equation*}
    \alg{A}_+ := (\SEG, \expl, \subsetsim, R_A).
  \end{equation*}
  Finally, for each $a \in A$ let $\Thet(a) = \{ (\ff{p}, \Gamma) \in \SEG \mid a \in \ff{p} \}$
  and $\Thet(A) = \{ \Thet(a) \mid a \in A \}$, and define
  \begin{equation*}
    \alg{A}_* := (\SEG, \expl, \subsetsim, R_A, \Thet(A)).
  \end{equation*}
\end{definition}

  It is clear that $\alg{A}_+ := (\SEG, \expl, \subsetsim, R_A)$ is a \CK-frame.
  To show that $\alg{A}_*$ is a general \CK-frame
  we must verify that $\Thet(A)$ is closed under the desired operations.
  We have $\Thet(\bot) = \{ \expl \} \in \Thet(A)$ and $\Thet(\top) = \SEG \in \Thet(A)$,
  and it easy to see that $\Thet(a) \cap \Thet(b) = \Thet(a \wedge b)$
  and $\Thet(a) \cup \Thet(b) = \Thet(a \vee b)$, so $\Thet(A)$ is closed under
  binary intersections and binary unions.
  Closure under $\To, \dbox$ and $\ddiamond$ will be verified in Lemma~\ref{lem:A}.
  
\begin{lemma}\label{lem:1}
  Let $\alg{A}$ be a \CK-algebra and $\ff{p} \in \pf(A)$.
  \begin{enumerate}
    \item $(\ff{p}, \{ A \})$ is a segment.
    \item If $\Diamond b \in \ff{p}$ and $\Diamond a \notin \ff{p}$,
          then there exists a prime filter $\ff{q}$ such that
          $\Box^{-1}(\ff{p}) \subseteq \ff{q}$ and $b \in \ff{q}$,
          but $a \notin \ff{q}$.
  \end{enumerate}
\end{lemma}
\begin{proof}
  The first item follows immediately from the definition of a segment.
  For the second, note that ${\downarrow}a := \{ c \in A \mid c \leq a \}$ is an ideal,
  and let $F$ be the smallest filter containing $\Box^{-1}(\ff{p})$ and $b$.
  We claim that $F \cap {\downarrow}a = \emptyset$.
  If this were not the case then, since $\Box^{-1}(\ff{p})$ is closed under finite meets,
  we could find $c \in \Box^{-1}(\ff{p})$ such that $c \wedge b \leq a$.
  Then $\Box c \wedge \Diamond b \leq \Diamond(c\wedge b) \leq \Diamond a$.
  But $\Box c \wedge \Diamond b\in\ff{p}$, which is upwards closed,
  so $\Diamond a \in \ff{p}$, a contradiction.
  So we can use the prime filter lemma~\ref{lem:prime}
  to extend $F$ to a prime filter $\ff{q}$ disjoint from ${\downarrow}a$.
\end{proof}

\begin{lemma}\label{lem:A}
  Let $\alg{A} = (A, \Box, \Diamond)$ be a \CK-algebra.
  Then for any $a, b \in A$ we have
  \begin{equation*}
    \Thet(a \to b) = \Thet(a) \To \Thet(b)
    \quad\text{and}\quad
    \Thet(\Box a) = \dbox \Thet(a)
    \quad\text{and}\quad
    \Thet(\Diamond a) = \ddiamond \Thet(a).
  \end{equation*}
\end{lemma}
\begin{proof}
  \textit{Case for implication.}
    Suppose $(\ff{p}, \Gamma) \in \Thet(a \to b)$, so $a \to b \in \ff{p}$.
    Then for any segment $(\ff{q}, \Delta)$ such that
    $(\ff{p}, \Gamma) \subsetsim (\ff{q}, \Delta)$
    we have $\ff{p} \subseteq \ff{q}$, so $a \to b \in \ff{q}$.
    If $(\ff{q}, \Delta) \in \Thet(a)$ then $a \in \ff{q}$ ,
    hence by deductive closure of $\ff{q}$ also $b \in \ff{q}$,
    so that $(\ff{q}, \Delta) \in \Thet(b)$.
    This proves that $(\ff{p}, \Gamma) \in \Thet(a) \To \Thet(b)$.
    
    Conversely, suppose $(\ff{p}, \Gamma) \notin \Thet(a \to b)$.
    Then $a \to b \notin \ff{p}$, so by the deduction theorem $\ff{p}, a \not\vdash b$.
    Therefore the deduction theorem yields a prime filter $\ff{q}$ that
    contains $\ff{p}$ and $a$ but not $b$, so $\ff{q} \in \Thet(a)$
    but $\ff{q} \notin \Thet(b)$.
    We can extend $\ff{q}$ to a segment, say, $(\ff{q}, \{ A \})$.
    Since $\ff{p} \subseteq \ff{q}$ we have
    $(\ff{p}, \Gamma) \subsetsim (\ff{q}, \{ A \})$, which witnesses
    $(\ff{p}, \Gamma) \notin \Thet(a) \To \Thet(b)$.
  
  \medskip\noindent
  \textit{Case for boxes.}
  Let $(\ff{p}, \Gamma)$ be a segment.
  Suppose $(\ff{p}, \Gamma) \in \Thet(\Box a)$, i.e.~$\Box a \in \ff{p}$.
  Let $(\ff{p}, \Gamma) \subsetsim (\ff{q}, \Delta) R_A (\ff{s}, \Sigma)$.
  Then $\Box a \in \ff{q}$ by definition of $\subsetsim$
  and $a \in \ff{s}$ because $(\ff{q}, \Delta)$ is a segment
  and $\ff{s} \in \Delta$ by definition of $R_A$. So $(\ff{s}, \Sigma) \in \Thet(a)$.
  This proves $(\ff{p}, \Gamma) \in \dbox \Thet(a)$.
  
  For the converse, suppose $(\ff{p}, \Gamma) \notin \Thet(\Box a)$,
  so that $\Box a \notin \ff{p}$.
  Then $\Box^{-1}(\ff{p})$ is a filter of $A$ that does not contain $a$,
  so we can extend it to a prime filter $\ff{q}$ that does not contain $a$.
  Then $(\ff{p}, \Gamma \cup \{ \ff{q} \})$ and
  $(\ff{q}, \{ A \})$ are segments.
  Now we have $(\ff{p}, \Gamma) \subsetsim (\ff{p}, \Gamma \cup \{ \ff{q} \}) R_A (\ff{q}, \{ A \})$
  and by construction $a \notin \ff{q}$, so $(\ff{q}, \{ A \}) \notin \Thet(a)$.
  Therefore $(\ff{p}, \Gamma) \notin \dbox \Thet(a)$.

  \medskip\noindent
  \textit{Case for diamonds.}
    Let $(\ff{p}, \Gamma)$ be a segment.
    Suppose $(\ff{p}, \Gamma) \in \Thet(\Diamond a)$, i.e.~$\Diamond a \in \ff{p}$.
    Let $(\ff{p}, \Gamma) \subsetsim (\ff{q}, \Delta)$.
    Then $\ff{p} \subseteq \ff{q}$ so $\Diamond a \in \ff{q}$,
    hence by definition of a segment there exists some $\ff{s} \in \Delta$ such
    that $a \in \ff{s}$. This implies that $(\ff{q}, \Delta) R_A (\ff{s}, \{ A \})$
    and $(\ff{s}, \{ A \}) \in \Thet(a)$.
    Therefore $\ff{p} \in \ddiamond\Thet(a)$.
    
    Conversely, suppose $(\ff{p}, \Gamma) \notin \Thet(\Diamond a)$,
    so $\Diamond a \notin \ff{p}$.
    Then for each $\Diamond b \in \ff{p}$ we can use Lemma~\ref{lem:1}
    to find a prime filter $\ff{q}_b$ containing $\Box^{-1}(\ff{p})$ and $b$ but
    not $a$.
    Let $\Theta := \{ \ff{q}_b \mid \Diamond b \in \ff{p} \}$.
    Then by construction $( \ff{p}, \Theta)$
    is a segment such that $(\ff{p}, \Theta) R_A (\ff{q}, \Delta)$
    implies $(\ff{q}, \Delta) \notin \Thet(a)$.
    Moreover $(\ff{p}, \Gamma) \subsetsim (\ff{p}, \Theta)$,
    hence $(\ff{p}, \Gamma) \notin \ddiamond(\Thet(a))$.
\end{proof}

  We conclude from Lemma~\ref{lem:A} that $\alg{A}_*$ is well defined.
  Furthermore, the map $\Thet$ witnesses:

\begin{proposition}
  The map $\Thet : \alg{A} \to (\alg{A}_*)^*$ is an isomorphism.
\end{proposition}

%===============================================================================
\subsection{Descriptive \CK-Frames}\label{subsec:descriptive}

  We have seen that the double dual of a \CK-algebra is isomorphic to itself.
  In this section we investigate when the same is true for general \CK-frames.
  That is, we identify conditions for a general \CK-frame $\mo{G}$ ensuring
  that it is isomorphic to $(\mo{G}^*)_*$.
  (An isomorphism between two general \CK-frames is
  a bijective function that preserves and reflects the inconsistent world,
  both orders, and the
  admissible sets.)

\begin{definition}
  Let $\mo{G} = (X, \expl, \leq, R, A)$ be a general \CK-frame.
  Then for $x \in X$ and $B \subseteq X$, we define
  $\eta(x) = \{ a \in A \mid x \in a \}$
  and $\eta[B] = \{ \eta(a) \mid a \in B \}$.
\end{definition}

\begin{lemma}
  Let $\mo{G} = (X, \expl, \leq, R, A)$ be a general \CK-frame and
  $(\mo{G}^*)_* = (\SEG, \expl, \subsetsim, R_A, \Thet(A))$ its double dual.
  Then the function
  \begin{equation*}
    \Eta
      : X \to \SEG
      : x \mapsto ( \eta(x), \eta[R[x]] )
  \end{equation*}
  is a well-defined order-preserving map $\mo{G} \to (\mo{G}^*)_*$.
\end{lemma}
\begin{proof}
  First we need to verify that it is well defined, i.e.~that
  $\Eta(x)$ is a segment for any $x \in X$.
  \begin{itemize}
    \item Suppose $\dbox a \in \eta(x)$.
          Then $x \in \dbox a$, so $x \leq y R z$ implies $z \in a$.
          In particular, $x R z$ implies $z \in a$, so $a \in \eta(z)$
          for all $\eta(z) \in \eta[R[x]]$.
    \item Suppose $\ddiamond a \in \eta(x)$.
          Then $x \in \ddiamond a$, so there exists some $z \in X$
          such that $x R z$ and $z \in a$. This implies that
          $\eta(z) \in \eta[R[x]]$ is such that $a \in \eta(z)$.
  \end{itemize}
  So $\Eta(x) = (\eta(x), \eta[R[x]])$ is a segment.
  
  Now let $x, y \in X$. If $x \leq y$ then $x \in a$ implies $y \in a$ for
  all $a \in A$, so $\eta(x) \subseteq \eta(y)$, and therefore $\Eta(x) \subsetsim \Eta(y)$.
  Lastly, suppose $x R y$. Then $\eta(y)$ is in the tail of $\Eta(x)$,
  and hence $\Eta(x) R_A \Eta(y)$.
\end{proof}

\begin{lemma}[Injectivity]\label{lem:Eta-inj}
  Let $\mo{G} = (X, \expl, \leq, R, A)$ be a general \CK-frame.
  If $\mo{G}$ satisfies:
  \begin{enumerate}
    \myitem{D1}\label{it:D1}
          for any $x,y \in X$, if $x \not\leq y$ then there exists $a \in A$ such that
          $x \in a$ and $y \notin a$;
  \end{enumerate}
  then $x \leq y$ iff $\Eta(x) \subsetsim \Eta(y)$.
  If moreover it satisfies
  \begin{enumerate}
    \myitem{D2}\label{it:D2}
          for all $x, y, z \in X$, if $x R y \sim z$ then $x R z$;
  \end{enumerate}
  then $\Eta$ is injective and an embedding with respect to $R$.
\end{lemma}
\begin{proof}
  We have already seen that $\Eta$ is order-preserving with respect to
  $\leq$ and $R$, so we only prove it to be order-reflecting.
  If $x \not\leq y$ then it follows immediately from~\eqref{it:D1} that
  $\eta(x) \not\subseteq \eta(y)$, hence $\Eta(x) \not\subsetsim \Eta(y)$.
  If $\Eta(x) R_A \Eta(y)$ then $\eta(y) \in \eta[R[x]]$,
  so $\eta(y) = \eta(z)$ for some $z \in R[x]$.
  This implies $x R z \sim y$, hence by~\eqref{it:D2} $x R y$.
\end{proof}

\begin{lemma}[Surjectivity]\label{lem:Eta-surj}
  Let $\mo{G} = (X, \expl, \leq, R, A)$ be a general \CK-frame.
  For each $a \in A$ define $-a = X \setminus a$, and let
  $-A = \{ -a \mid a \in A \}$.
  If $\mo{G}$ satisfies:
  \begin{enumerate}
    \myitem{D3}\label{it:D3}
          if $B \subseteq A \cup -A$ has the finite intersection property,
          then $\bigcap B \neq \emptyset$;
  \end{enumerate}
  then for every $(\ff{p}, \Gamma) \in \SEG$ there exists an $x \in X$ such that
  $x \sim (\ff{p}, \Gamma)$.
  Property~\eqref{it:D3} is also known as \emph{compactness}.
  If moreover,
  \begin{enumerate}
    \myitem{D4}\label{it:D4}
          if $x \in X$ and $U \subseteq X$ is such that for all $a \in A$:
          \begin{itemize}
            \item $x \in \dbox a$ implies $U \subseteq a$,
            \item $x \in \ddiamond a$ implies $U \cap a \neq \emptyset$,
          \end{itemize}
          then there exists a $x' \in X$ such that $x \sim x'$ and $R[x'] = U$;
  \end{enumerate}
  then $\Eta$ is surjective.
\end{lemma}
\begin{proof}
  Suppose $\mo{G}$ satisfies~\eqref{it:D3}
  and let $(\ff{p}, \Gamma) \in \SEG$.
  We need to find an element $x$ such that $x \in a$ if and only if $a \in \ff{p}$.
  If $\ff{p} = \expl$ (in $\SEG$) then we can take $x = \expl$ (in $\mo{G}$),
  so we may assume that $\ff{p}$ is a proper prime filter.
  Let $\ms{C} = \ff{p} \cup \{ -b \mid b \in (A \setminus \ff{p}) \}$.
  Then we seek some $x \in \bigcap \ms{C}$.
  By~\eqref{it:D3} it suffices to prove that $\ms{C}$ has the finite intersection property.
  For any finite $B_1 \subseteq \ff{p}$ we have $\bigcap B_1 \in \ff{p}$
  and $\bigcap B_1 \neq \emptyset$ because $\ff{p}$ is a proper filter.
  Similarly, for any finite $B_2 \subseteq (A \setminus \ff{p})$ we have
  $\bigcup B_2 \notin \ff{p}$ because $\ff{p}$ is prime,
  hence $\{ -b \mid b \in A \setminus \ff{p} \}$ has the finite intersection property.
  So it suffices to prove that for any $a \in \ff{p}$ and $b \in (A \setminus \ff{p})$
  we have
  \begin{equation*}
    a \cap -b \neq \emptyset.
  \end{equation*}
  Suppose towards a contradiction that $a \cap -b = \emptyset$.
  Then $a \subseteq b$, so that $a \in \ff{p}$ implies $b \in \ff{p}$, a contradiction.
  This proves that $\ms{C}$ has the finite intersection property,
  hence by~\eqref{it:D3} $\bigcap \ms{C} \neq \emptyset$ and by construction we
  have $\eta(x) = \ff{p}$, so $\Eta(x) \sim (\ff{p}, \Gamma)$, for any $x \in \bigcap \ms{C}$.
  
  Next, assume that additionally~\eqref{it:D4} holds and let $(\ff{p}, \Gamma) \in \SEG$.
  Let $x \in X$ be such that $\eta(x) = \ff{p}$,
  and $U := \{ y \in X \mid \eta(y) \in \Gamma \}$.
  Then $\eta[U] = \Gamma$ by~\eqref{it:D3}.
  Moreover, if $a \in A$ and $x \in \dbox a$ then
  $\dbox a \in \ff{p}$ so $a \in \eta(y)$ for every $y \in U$,
  hence $y \in a$. Therefore $U \subseteq a$.
  Also, if $\ddiamond a \in \ff{p}$ then there exists a $\eta(y) \in \Gamma$
  such that $a \in \eta(y)$, i.e.~$y \in a$.
  Therefore $a \cap U \neq \emptyset$.
  So by~\eqref{it:D4} there exists $x' \in X$ such that $x \sim x'$ and $R[x'] = U$,
  and by design $\Eta(x') = (\eta(x), \eta[R[x]]) = (\ff{p}, \eta[U]) = (\ff{p}, \Gamma)$.
  So $\Eta$ is surjective.
\end{proof}

\begin{definition}
  A \emph{descriptive \CK-frame} is a general \CK-frame that satisfies
  \eqref{it:D1}, \eqref{it:D2}, \eqref{it:D3} and~\eqref{it:D4}.
\end{definition}

\begin{lemma}\label{lem:A+descr}
  If $\alg{A}$ is a \CK-algebra, then $\alg{A}_*$ is a descriptive \CK-frame.
\end{lemma}
\begin{proof}
  We know that $\alg{A}_*$ is a general \CK-frame, so we only have to verify
  that it satisfies~\eqref{it:D1} up to~\eqref{it:D4}.
  \begin{enumerate}
    \item[\eqref{it:D1}]
              Suppose $(\ff{p}, \Gamma) \not\subsetsim (\ff{q}, \Delta)$.
              Then $\ff{p} \not\subseteq \ff{q}$, so there
              exists some $a \in A$ such that $a \in \ff{p}$ but $a \notin \ff{q}$.
              This implies that $\Thet(a) \in \Thet(A)$ is such that
              $(\ff{p}, \Gamma) \in \Thet(a)$ and $(\ff{q}, \Delta) \notin \Thet(a)$.
    \item[\eqref{it:D2}] 
              Suppose $(\ff{p}, \Gamma) R_A (\ff{q}, \Delta) \sim (\ff{q}', \Delta')$.
              Then $\ff{q} \in \Gamma$ and $\ff{q} = \ff{q}'$,
              hence $(\ff{p}, \Gamma) R_A (\ff{q}', \Delta')$.
    \item[\eqref{it:D3}]
              Suppose $B \subseteq \Thet(A) \cup -\Thet(A)$ has the finite
              intersection property.
              We start by constructing a prime filter in $\bigcap B$.
              To this end, define sets $F, I \subseteq A$ by
              \begin{align*}
                F &= {\uparrow} \{ a_1 \wedge \cdots \wedge a_n 
                  \mid n \in \mathbb N \text{ and } \Thet(a_1), \ldots, \Thet(a_n) \in B \} \\
                I &= {\downarrow} \{ c_1 \vee \cdots \vee c_n \in A
                  \mid n \in \mathbb N \text{ and } -\Thet(c_1), \ldots, -\Thet(c_n) \in B \}.
              \end{align*}
              Then $F$ is a filter and $I$ is an ideal, by construction.
              Moreover, $F \cap I = \emptyset$, for if this were not the case then
              there would exist
              $\Thet(a_1), \ldots, \Thet(a_n), -\Thet(c_1), \ldots, -\Thet(c_m) \in B$
              such that $a_1 \wedge \cdots \wedge a_n \leq -c_1 \vee \cdots \vee -c_m$,
              hence
              \begin{equation*}
                \Thet(a_1) \cap \cdots \cap \Thet(a_n)
                  \cap -\Thet(c_1) \cap \cdots \cap -\Thet(c_m) = \emptyset,
              \end{equation*}
              contradicting the assumption that $B$ has the finite intersection property.
              So we can use the prime filter lemma to find a prime filter $\ff{p}$
              containing $F$ which is disjoint from $I$.
              By construction any segment of the form $(\ff{p}, \Delta)$ is in $\bigcap B$.
              So extending $\ff{p}$ to a segment, say, $(\ff{p}, \{ \expl \})$,
              proves $\bigcap B \neq \emptyset$.
    \item[\eqref{it:D4}]
               Let $(\ff{p}, \Gamma) \in \SEG$ and $U \subseteq \SEG$ be such that
              the given conditions are satisfied.
              Then by definition $(\ff{p}, U)$ is also a segment,
              and by construction $R_A[(\ff{p}, U)] = U$.
  \end{enumerate}
  So $\alg{A}_*$ is descriptive.
\end{proof}

\begin{proposition}
  Let $\mo{G}$ be a general \CK-frame. Then $\mo{G} \cong (\mo{G}^*)_*$ if and
  only if $\mo{G}$ is descriptive.
\end{proposition}
\begin{proof}
  If $\mo{G}$ is descriptive then $\Eta : \mo{G} \to (\mo{G}^*)_*$ is an
  isomorphism by Lemmas~\ref{lem:Eta-inj} and~\ref{lem:Eta-surj}.
  Conversely, $(\mo{G}^*)_*$ is descriptive by Lemma~\ref{lem:A+descr} so
  if $\mo{G} \cong (\mo{G}^*)_*$ then $\mo{G}$ is also descriptive.
\end{proof}

%===============================================================================
\subsection{Full Duality}\label{subsec:dual}
  
  The previous two subsections established a duality between the class of
  \CK-algebras and that of descriptive \CK-frames.
  We now extend this to a categorical duality by extending the duality
  to morphisms.

\begin{definition}
  A \emph{general bounded morphism} from between general \CK-frames
  $(X, \expl, \leq, R, A)$ and
  $(X', \expl', \leq', R', A')$ is a bounded morphism
  $f : (X, \expl, \leq, R) \to (X', \expl', \leq', R')$ such that
  $f^{-1}(a') \in A$ for all $a' \in A'$.
\end{definition}

\begin{lemma}\label{lem:gbm-to-homom}
  Let $\mo{G} = (X, \expl, \leq, R, A)$ and $\mo{G}' = (X', \expl', \leq', R', A')$
  are two general \CK-frames and $f : \mo{G} \to \mo{G}'$ a bounded morphism.
  Then $f^* = f^{-1} : (\mo{G}')^* \to \mo{G}^*$ is a homomorphism of \CK-algebras.
\end{lemma}
\begin{proof}
  Clearly $f^{-1}(A') = A$ so $f^*$ preserves the top element,
  and~\eqref{it:Bt} implies that $f^*$ also preserves the bottom element.
  Meets and joins in complex algebras are given by intersections and unions,
  which are preserved by the nature of the inverse image.
  The proof that $f^*$ also preserves $\To$, $\dbox$ and $\ddiamond$
  is analogues to the proof of Lemma~\ref{lem:bm_props}\eqref{it:mor-cpx-alg}.
\end{proof}

  So we can define a contravariant functor
  \begin{equation*}
    (\cdot)^* : \CKDescr \to \CKAlg
  \end{equation*}
  by sending a descriptive \CK-frame $\mo{G}$ to $\mo{G}^*$ and
  a bounded morphism $f$ to $f^*$.
  For a functor in the opposite direction we have to work a bit harder:

\begin{lemma}
  Let $\alg{A} = (A, \Box, \Diamond)$ and $\alg{A}' = (A', \Box', \Diamond')$
  be two \CK-algebras with dual descriptive \CK-frames
  $\alg{A}_* = (\SEG, \expl, \subsetsim, R, \Thet(A))$
  and $\alg{A}'_* = (\SEG', \expl', \subsetsim', R', \Thet(A'))$.
  If $h : \alg{A} \to \alg{A}'$ is a homomorphism,
  then
  \begin{equation*}
    h_*
      : \alg{A}'_* \to \alg{A}_*
      : (\ff{p}', \Gamma') \mapsto (h^{-1}(\ff{p}'), h^{-1}[\Gamma'])
  \end{equation*}
  is a general bounded morphism.
\end{lemma}
\begin{proof}
  We split the lemma into five steps.
  
  \medskip\noindent
  \textit{Step 1: $h_*$ is well defined.}
    We need to verify that $h_*(\ff{p}', \Gamma') = (h^{-1}(\ff{p}'), h^{-1}[\Gamma']) \in \SEG$
    for any $(\ff{p}', \Gamma') \in \SEG'$. We go over the two items defining a segment:
    \begin{itemize}
      \item If $\Box a \in h^{-1}(\ff{p}')$ then $\Box' h(a) = h(\Box a) \in \ff{p}'$,
            so $h(a) \in \ff{q}'$ for all $\ff{q}' \in \Gamma'$, because
            $(\ff{p}', \Gamma')$ is a segment. Therefore $a \in h^{-1}(\ff{q}')$.
            Since every element of $h^{-1}[\Gamma']$ is of the form $h^{-1}(\ff{q}')$ for
            some $\ff{q}' \in \Gamma'$, this shows that $a \in \ff{r}$ for all
            $\ff{r} \in h^{-1}[\Gamma']$.
      \item If $\Diamond a \in h^{-1}(\ff{p}')$ then $\Diamond' h(a) = h(\Diamond a) \in \ff{p}'$,
            so there exists some $\ff{q}' \in \Gamma'$ such that $h(a) \in \ff{q}'$.
            This implies $a \in h^{-1}(\ff{q}') \in h^{-1}[\Gamma']$.
    \end{itemize}
    So $h_*(\ff{p}', \Gamma')$ is indeed a segment.
  
  \medskip\noindent
  \textit{Step 2: $h_*(\ff{p}', \Gamma') = \expl$ iff $(\ff{p}', \Gamma') = \expl'$.}
    Clearly $h^{-1}(A') = A$, so $h_*(\expl') = h_*(A', \{ A' \}) = (A, \{ A \}) = \expl$.
    Conversely, suppose $h_*(\ff{p}', \Gamma') = (A, \{ A \})$.
    Then $A = h^{-1}(\ff{p}')$, so $h(\bot) = \bot' \in \ff{p}'$, hence $\ff{p}' = A'$.
    Similarly, for any $\ff{q}' \in \Gamma'$ we have
    $h^{-1}(\ff{q}') = A$, so that $\ff{q}' = A'$.
    So $(\ff{p}', \Gamma') = (A', \{ A' \}) = \expl'$.

  \medskip\noindent
  \textit{Step 3: $h_* : (\SEG', \subsetsim') \to (\SEG, \subsetsim)$ is a bounded morphism.}
    Suppose $(\ff{p}', \Gamma') \subsetsim' (\ff{q}', \Delta')$.
    Then $\ff{p}' \subseteq \ff{q}'$, hence $h^{-1}(\ff{p}') \subseteq h^{-1}(\ff{q}')$,
    so that $h_*(\ff{p}', \Gamma') \subsetsim h_*(\ff{q}', \Delta')$.
    
    Now suppose $h_*(\ff{p}', \Gamma') \subsetsim (\ff{s}, \Sigma)$
    for some segment $(\ff{s}, \Sigma) \in \SEG$.
    Then $h^{-1}(\ff{p}') \subseteq \ff{s}$.
    We need to find a segment $(\ff{s}', \Sigma')$ such that
    $\ff{p}' \subseteq \ff{s}'$ and $h^{-1}(\ff{s}') = \ff{s}$.
    Since there are no conditions on $\Sigma'$, we can focus on finding
    a suitable prime filter $\ff{s}'$.
    If $\ff{p}' = A'$ or $\ff{s} = A$ then it can easily be verified that we
    can take $\ff{s}' = A'$,
    so assume that $\ff{p}'$ and $\ff{s}$ are proper prime filters.
    Then we need to find some $\ff{s}'$ in the intersection of the following
    collection of sets:
    \begin{equation*}
      \underbrace{\{ \theta(a') \mid a' \in \ff{p}' \}}_{\ms{C}_1}
        \cup \underbrace{\{ \theta(h(b)) \mid b \in \ff{s} \}}_{\ms{C}_2}
        \cup \underbrace{\{ \pf(A') \setminus \theta(h(c)) \mid c \in (A \setminus \ff{s}) \}}_{\ms{C}_3}
    \end{equation*}
    Indeed, $\ff{s}' \in \bigcap \ms{C}_1$ implies $\ff{p}' \subseteq \ff{s}'$,
    $\ff{s}' \in \bigcap \ms{C}_2$ implies $\ff{s} \subseteq h^{-1}(\ff{s}')$,
    and $\ff{s}' \in \bigcap \ms{C}_3$ implies $h^{-1}(\ff{s}') \subseteq \ff{s}$.
    By compactness of $\alg{A}'_*$, it suffices to prove that
    $\ms{C}_1 \cup \ms{C}_2 \cup \ms{C}_3$ has the finite intersection property.
    Since $\ff{p}'$ and $\ff{s}$ are proper prime filters,
    each of $\ms{C}_1, \ms{C}_2$ and $\ms{C}_3$ is
    closed under finite intersections, so it suffices to prove that for
    any $a' \in \ff{p}'$, $b \in \ff{s}$ and $c \in (A \setminus \ff{s})$ we have
    \begin{equation*}
      \theta(a') \cap \theta(h(b)) \cap (\pf(A') \setminus \theta(h(c))) \neq \emptyset.
    \end{equation*}
    Suppose towards a contradiction that this is not the case.
    Then
    \begin{equation*}
      \theta(a') \cap \theta(h(b)) \subseteq \theta(h(c)),
    \end{equation*}
    and hence
    \begin{equation*}
      \theta(a')
        \subseteq \theta(h(b)) \To \theta(h(c))
        = \theta(h(b) \to h(c))
        = \theta(h(b \to c)).
    \end{equation*}
    This implies $a' \leq' h(b \to c)$.
    Since $a' \in \ff{p}'$ this implies $h(b \to c) \in \ff{p}'$ and hence
    $b \to c \in h^{-1}(\ff{p}') \subseteq \ff{s}$.
    But then the assumption that $b \in \ff{s}$ implies $c \in \ff{s}$, a contradiction.
    We conclude that $\ms{C}_1 \cup \ms{C}_2 \cup \ms{C}_3$ has the finite
    intersection property, hence we can find a suitable $\ff{s}'$ in its intersection.
    Extending this to a segment, say, $(\ff{s}', \{ A' \})$ yields the
    boundedness property.

  \medskip\noindent
  \textit{Step 4: $h_* : (\SEG', R') \to (\SEG, R)$ is a bounded morphism.}
    Suppose $(\ff{p}', \Gamma') R' (\ff{q}', \Delta')$.
    Then $\ff{q}' \in \Gamma'$, so $h^{-1}(\ff{q}') \in h^{-1}[\Gamma']$,
    hence $h_*(\ff{p}', \Gamma') R h_*(\ff{q}', \Delta')$.

    Now suppose $h_*(\ff{p}', \Gamma') R (\ff{s}, \Sigma)$.
    We need to find a segment $(\ff{s}', \Sigma')$ such that
    $(\ff{p}', \Gamma') R' (\ff{s}', \Sigma')$ and $h_*(\ff{s}', \Sigma') = (\ff{s}, \Sigma)$.
    By assumption $\ff{s} \in h^{-1}[\Gamma']$,
    so there exists some $\ff{s}' \in \Gamma'$ such that $\ff{s} = h^{-1}(\ff{s}')$.
    By construction, any segment of the form $(\ff{s}', \Sigma')$ is such that
    $(\ff{p}', \Gamma') R' (\ff{s}', \Sigma')$, so we are left to construct $\Sigma'$
    such that $h^{-1}[\Sigma'] = \Sigma$.

    If $\ff{s}' = A'$ then $h(\bot) = \bot' \in \ff{s}'$
    hence $\bot \in h^{-1}(\ff{s}') = \ff{s}$, so that $\ff{s} = A$.
    Then we must have $(\ff{s}, \Sigma) = (A, \{ A \})$,
    and taking $(\ff{s}', \Sigma') = (A', \{ A' \})$ gives the desired result.

    Now assume $\ff{s}' \neq A'$, so $\ff{s}'$ is proper, and let
    \begin{equation*}
      \Sigma' = \{ \ff{r}' \in \pf(A') \mid h^{-1}(\ff{r}') \in \Sigma 
                   \text{ and } \Box^{-1}(\ff{s}') \subseteq \ff{r}' \}.
    \end{equation*}
    (The second part simply states that $\Box' a' \in \ff{s}'$ implies
    $a' \in \ff{r}'$ for all $a' \in A'$.)
    Let us verify that this yields a segment.
    By construction $\Box' a' \in \ff{s}'$ implies $a' \in \ff{r}'$ for all $\ff{r}' \in \Sigma'$.
    If $A \in \Sigma$ then $A' \in \Sigma'$, and we automatically have
    that $\Diamond' a' \in \ff{s}'$ implies $a' \in \ff{r}'$ for some $\ff{r}' \in \Sigma'$
    (namely $\ff{r}' = A'$).
    So assume $A \notin \Sigma$ and $\Diamond' a' \in \ff{s}'$.
    We need to find some $\ff{r}' \in \Sigma'$ that contains $a'$,
    i.e.~some $\ff{r}'$ in the intersection of
    \begin{equation*}
      \ms{B} := \{ \theta(b') \mid \Box'b' \in \ff{s}' \} \cup \{ \theta(a') \}.
    \end{equation*}
    To prove that $\bigcap\ms{B}$ is nonempty,
    by compactness it suffices to prove that $\ms{B}$ has the finite
    intersection property. The set $\{ \theta(b') \mid \Box'b' \in \ff{s}' \}$
    has the finite intersection property because $\Box'$ distributes over
    meets and $\ff{s}'$ is a proper filter, so it suffices to prove that for any
    $\Box'b' \in \ff{s}'$ we have
    \begin{equation*}
      \theta(b') \cap \theta(a') \neq \emptyset.
    \end{equation*}
    Suppose the contrary, then $b' \wedge a' \leq \bot'$,
    hence $\Box' b' \wedge \Diamond a' \leq \Diamond' \bot' \in \ff{s}'$.
    Since $h(\Diamond\bot) = \Diamond'\bot'$ and $\ff{s} = h^{-1}(\ff{s}')$ this
    implies $\Diamond\bot \in \ff{s}$, so that the diamond-condition of a segment
    forces $A \in \Sigma$, a contradiction.
    It follows that $\ms{B}$ has the finite intersection property,
    hence we can find a suitable $\ff{r}' \in \bigcap\ms{B}$
    witnessing the fact that $(\ff{s}', \Sigma')$ is a segment.
    
    Lastly, in order to show that $h^{-1}[\Sigma'] = \Sigma$, we need to show that
    for each $\ff{r} \in \Sigma$ there exists some $\ff{r}' \in \Sigma'$ such that
    $h^{-1}(\ff{r}') = \ff{r}$. So pick such $\ff{r} \in \Sigma$.
    If either $\ff{s}' = A'$ or $\ff{r} = A$ then we can take $\ff{r}' = A'$,
    so we may assume that both $\ff{s}'$ and $\ff{r}$ are proper prime filters.
    We need to find some $\ff{r}'$ in the intersection of
    \begin{equation*}
      \underbrace{\{ \theta(a') \mid \Box'a' \in \ff{s}' \}}_{\ms{C}_1}
        \cup \underbrace{\{ \theta(h(b)) \mid b \in \ff{r} \}}_{\ms{C}_2}
        \cup \underbrace{\{ \pf(A') \setminus \theta(h(c)) \mid c \in (A \setminus \ff{r}) \}}_{\ms{C}_3}
    \end{equation*}
    As before, it suffices to prove that $\ms{C}_1 \cup \ms{C}_2 \cup \ms{C}_3$
    has the finite intersection property. 
    Since $\ff{s}'$ and $\ff{r}$ are proper prime filters, each of the $\ms{C}_i$ is
    closed under finite intersections. So it suffices to prove that for any
    $\Box' a' \in \ff{s}'$, $b \in \ff{r}$ and $c \in (A \setminus \ff{r})$,
    \begin{equation*}
      \theta(a') \cap \theta(h(b)) \cap (\pf(A') \setminus \theta(h(c))) \neq \emptyset.
    \end{equation*}
    Suppose towards a contradiction that this is not the case.
    Then as before we obtain $a' \leq' h(b \to c)$,
    and hence $\Box' a' \leq' \Box'(h(b \to c)) = h(\Box(b \to c))$.
    This implies $\Box(b \to c) \in h^{-1}(\ff{p}')$, and hence $b \to c \in \ff{s}$.
    Since $b \in \ff{s}$ we find $c \in \ff{s}$, a contradiction.
    So the given intersection is nonempty, which implies the existence of some
    $\ff{r}' \in \Sigma'$ such that $\ff{r} = h^{-1}(\ff{r}')$,
    hence $h^{-1}[\Sigma'] = \Sigma$.
    
    Thus, the segment $(\ff{s}', \Sigma')$ is such that
    $(\ff{p}', \Gamma') R' (\ff{s}', \Sigma')$ and $h_*(\ff{s}', \Sigma') = (\ff{s}, \Sigma)$,
    which proves the boundedness condition for the modal accessibility relation.
    Ultimately, this proves that $h_*$ is a bounded morphism.

  \medskip\noindent
  \textit{Step 5: $(h_*)^{-1}(\Thet(a)) \in \Thet(A')$ for any $\Thet(a) \in \Thet(A)$.}
    We conclude the lemma by proving that $h_*$ is a \emph{general} bounded
    morphism. To this end, observe that any $\Thet(a) \in \Thet(A)$ we have
    \begin{align*}
      (h_*)^{-1}(\Thet(a))
        &= \{ (\ff{p}', \Gamma') \in \SEG' \mid h_*(\ff{p}', \Gamma') \in \Thet(a) \} \\
        &= \{ (\ff{p}', \Gamma') \in \SEG' \mid (h^{-1}(\ff{p}'), h^{-1}[\Gamma']) \in \Thet(a) \} \\
        &= \{ (\ff{p}', \Gamma') \in \SEG' \mid a \in h^{-1}(\ff{p}') \} \\
        &= \{ (\ff{p}', \Gamma') \in \SEG' \mid h(a) \in \ff{p}' \} \\
        &= \Thet(h(a))
    \end{align*}
    The latter is in $\Thet(A')$ as required.
    This completes the proof of the lemma.
\end{proof}

  We can now verify that
  \begin{equation*}
    (\cdot)_* : \CKAlg \to \CKDescr
  \end{equation*}
  is a contravariant function.

\begin{theorem}
  The contravariant functors $(\cdot)^*$ and $(\cdot)_*$ establish a dual
  equivalence
  \begin{equation*}
    \CKAlg \equiv^{\op} \CKDescr.
  \end{equation*}
\end{theorem}
\begin{proof}
  We need to prove that $\Eta : \id_{\CKDescr} \to (\cdot)_* \circ (\cdot)^*$
  and $\Thet : \id_{\CKAlg} \to (\cdot)^* \circ (\cdot)_*$ are natural
  isomorphisms that satisfy the triangle identities,
  i.e.~for any descriptive \CK-frame $\mo{X}$ and \CK-algebra $\alg{A}$:
  \begin{equation*}
    \begin{tikzpicture}[xscale=2]
      %% nodes
        \node (0) at (0,0) {$\mo{X}^*$};
        \node (1) at (1,0) {$((\mo{X}^*)_*)^*$};
        \node (2) at (2,0) {$\mo{X}^*$};
      %% arrows
        \draw[-latex] (0) to node[above]{\footnotesize{$\Thet_{\mo{X}^*}$}} (1);
        \draw[-latex] (1) to node[above]{\footnotesize{$(\Eta_{\mo{X}})^*$}} (2);
        \draw[-latex, bend right=40] (0) to node[below]{\footnotesize{$\id_{\mo{X}^*}$}} (2);
      %% nodes
        \node (3) at (3,0) {$\alg{A}_*$};
        \node (4) at (4,0) {$((\alg{A}_*)^*)_*$};
        \node (5) at (5,0) {$\alg{A}_*$};
      %% arrows
        \draw[-latex] (3) to node[above]{\footnotesize{$\Eta_{\alg{A}_*}$}} (4);
        \draw[-latex] (4) to node[above]{\footnotesize{$(\Thet_{\alg{A}})_*$}} (5);
        \draw[-latex, bend right=40] (3) to node[below]{\footnotesize{$\id_{\alg{A}_*}$}} (5);
    \end{tikzpicture}
  \end{equation*}
  We have already seen that $\Eta$ and $\Thet$ are isomorphisms on components,
  so we only have to verify naturality and the triangle identities.

  \medskip\noindent
  \textit{Claim 1: $\Eta$ is natural.}
    We need to show that for any bounded morphism $f : \mo{G} \to \mo{G}'$
    the following diagram commutes:
    \begin{equation*}
      \begin{tikzpicture}[yscale=.7]
        %% nodes
          \node (G)   at (0,2) {$\mo{G}$};
          \node (Gp)  at (0,0) {$\mo{G}'$};
          \node (Gd)  at (2,2) {$(\mo{G}^*)_*$};
          \node (Gpd) at (2,0) {$((\mo{G}')^*)_*$};
        %% edges
          \draw[-latex] (G) to node[left]{\footnotesize{$f$}} (Gp);
          \draw[-latex] (Gd) to node[right]{\footnotesize{$(f^*)_*$}} (Gpd);
          \draw[-latex] (G) to node[above]{\footnotesize{$\Eta_{\mo{G}}$}} (Gd);
          \draw[-latex] (Gp) to node[below]{\footnotesize{$\Eta_{\mo{G}'}$}} (Gpd);
      \end{tikzpicture}
    \end{equation*}
    Let $x$ be a world in $\mo{G}$. Then we need to show that
    \begin{equation}\label{eq:Eta-natural}
      (\eta_{\mo{G}'}(f(x)), \eta_{\mo{G}'}[R'[f(x)]])
        = \big(((f^{-1})^{-1} \circ \eta_{\mo{G}})(x), ((f^{-1})^{-1} \circ \eta_{\mo{G}})[R[x]] \big).
    \end{equation}
    
    First observe that for any $y \in X$ we have
    \begin{align*}
      a' \in (f^{-1})^{-1}(\eta_{\mo{G}}(y))
        \iff f^{-1}(a') \in \eta(y)
        \iff y \in f^{-1}(a')
        \iff f(y) \in a'
        \iff a' \in \eta_{\mo{G}'}(f(y)).
    \end{align*}
    and hence $(f^{-1})^{-1}(\eta_{\mo{G}}(y)) = \eta_{\mo{G}'}(f(y))$.
    This immediately proves that the first coordinates of the segments
    in~\eqref{eq:Eta-natural} coincide.
    For the second argument
    suppose $(\ff{s}', \Sigma') \in \eta_{\mo{G}'}[R'[f(z)]]$.
    Then there exists some $z' \in R'[f(x)]$ such that
    $(\ff{s}', \Sigma') = \eta_{\mo{G}'}(z')$.
    Since $f$ is a bounded morphism, there exists some $z \in X$
    such that $xRz$ and $f(z) = z'$, and hence
    \begin{equation*}
      (\ff{s}', \Sigma')
        = \eta_{\mo{G}'}(f(z))
        = (f^{-1})^{-1}(\eta_{\mo{G}}(z))
        \in ((f^{-1})^{-1} \circ \eta_{\mo{G}})[R[x]]
    \end{equation*}
    This implies
    $\eta_{\mo{G}'}[R'[f(x)]] \subseteq ((f^{-1})^{-1} \circ \eta_{\mo{G}})[R[x]]$.
    For the converse, suppose $(\ff{s}', \Sigma') \in ((f^{-1})^{-1} \circ \eta_{\mo{G}})[R[x]]$.
    Then there exists some $z \in R[x]$ such that $(\ff{s}', \Sigma') = 
    (f^{-1})^{-1}(\eta_{\mo{G}}(z)) = \eta_{\mo{G}'}(f(z))$.
    Since $xRz$ we have $f(x) R' f(z)$, so $f(z) \in R'[f(x)]$, and therefore
    $(\ff{s}', \Sigma') \in \eta_{\mo{G}'}[R'[f(x)]]$,
    proving the other inclusion.
    We conclude that the equality in~\eqref{eq:Eta-natural} is true.

  \medskip\noindent
  \textit{Claim 2: $\Thet$ is natural.}
    Let $h : \alg{A} \to \alg{B}$ be a homomorphism between \CK-algebras.
    We need to prove that $\Thet_{\alg{A}'} \circ h = (h_*)^* \circ \Thet_{\alg{A}}$.
    Let $a \in A$ and let $(\ff{p}', \Gamma)$ be a segment in $\alg{A}_*$.
    Then we have
    \begin{alignat*}{2}
      (\ff{p}', \Gamma) \in \Thet_{\alg{A}'}(h(a))
        &\iff h(a) \in \ff{p}'
        &&\iff a \in h^{-1}(\ff{p}') \\
        &\iff h_*(\ff{p}, \Gamma) \in \Thet_{\alg{A}}(a)
        &&\iff (\ff{p}', \Gamma) \in (h_*)^*(\Thet_{\alg{A}}(a))
    \end{alignat*}
    so $\Thet_{\alg{A}'}(h(a)) = (h_*)^*(\Thet_{\alg{A}}(a))$.
    This proves the desired equality.

  \medskip\noindent
  \textit{Claim 3: The left hand triangle identity is satisfied.}
  Let $\mo{D} = (X, \expl, \leq, R, A)$ be a descriptive
  \CK-frame. 
  Then for all $a \in A$ and $x \in X$ we have
  \begin{alignat*}{2}
    x \in (\Eta_{\mo{D}})^*(\Thet_{\mo{D}^*}(a))
      &\iff \Eta_{\mo{D}}(x) \in \Thet_{\mo{D}^*}(a)
      &&\iff (\eta_{\mo{D}}(x), \eta_{\mo{D}}[R[x]]) \in \Thet_{\mo{D}^*}(a) \\
      &\iff a \in \eta_{\mo{D}}(x)
      &&\iff x \in a
  \end{alignat*}
  This implies that the first triangle identity holds.

  \medskip\noindent
  \textit{Claim 4: The right hand triangle identity is satisfied.}
  For the second one, let $\alg{A} = (A, \Box, \Diamond)$ be a \CK-algebra.
  First we observe that that for any segment $(\ff{q}, \Delta)$ in $\alg{A}_*$ we have
  \begin{equation}\label{eq:dual-1}
    (\Thet_{\alg{A}})^{-1}(\eta_{\alg{A}_*}(\ff{q}, \Delta)) = \ff{q}
    \quad\text{and}\quad
    ((\Thet_{\alg{A}})^{-1} \circ \eta_{\alg{A}_*})[R_A[(\ff{p},\Gamma)]] = \Gamma.
  \end{equation}
  The first equality follows from
  \begin{equation*}
    a \in (\Thet_{\alg{A}})^{-1}(\eta_{\alg{A}_*}(\ff{q}, \Delta))
      \iff \Thet_{\alg{A}}(a) \in \eta_{\alg{A}_*}(\ff{q}, \Delta)
      \iff (\ff{q}, \Delta) \in \Thet_{\alg{A}}
      \iff a \in \ff{q}.
  \end{equation*}
  For the second, if $\ff{q} \in \Gamma$ then we can
  extend $\ff{q}$ to a segment, say, $(\ff{q}, \{ A \})$ that is in
  $R_A[(\ff{p}, \Gamma)]$, and hence 
  \begin{equation*}
    \ff{q}
      = ((\Thet_{\alg{A}})^{-1} \circ \eta_{\alg{A}_*})(\ff{q}, \{ A \})
      \in ((\Thet_{\alg{A}})^{-1} \circ \eta_{\alg{A}_*})[R_A[(\ff{p},\Gamma)]].
  \end{equation*}
  Conversely, if
  $\ff{q} \in ((\Thet_{\alg{A}})^{-1} \circ \eta_{\alg{A}_*})[R_A[(\ff{p},\Gamma)]]$
  then there must be a segment of the form $(\ff{q}, \Delta)$
  in $R_A[(\ff{p},\Gamma)]$, which entails $\ff{q} \in \Gamma$.

  Now let $(\ff{p}, \Gamma)$ be a segment in $\alg{A}_*$ and compute
  \begin{align*}
    (\Thet_{\alg{A}})_*(\Eta_{\alg{A}_*}(\ff{p}, \Gamma))
      &= (\Thet_{\alg{A}})_*(\eta_{\alg{A}_*}(\ff{p}, \Gamma), \eta_{\alg{A}_*}[R_A[(p,\Gamma)]]) \\
      &= ((\Thet_{\alg{A}})^{-1}(\eta_{\alg{A}_*}(\ff{p}, \Gamma)), ((\Thet_{\alg{A}})^{-1} \circ \eta_{\alg{A}_*})[R_A[(\ff{p},\Gamma)]]) \\
      &= (\ff{p}, \Gamma)
  \end{align*}
  The last equality follows from~\eqref{eq:dual-1}.
\end{proof}

  As usual, the connection between $\CK$-algebras and descriptive $\CK$-frames
  gives:

\begin{theorem}
  Any extension $\CK$ with axioms $\mathsf{Ax}$ is sound and complete with
  respect to the class of descriptive frames that validates the axioms in $\mathsf{Ax}$.
\end{theorem}

%%%%%%%%%%%%%%%%%%%%%%%%%%%%%%%%%%%%%%%%%%%%%%%%%%%%%%%%%%%%%%%%%%%%%%%%%%%%%%%%
\section{Sahlqvist Correspondence and Completeness}\label{Sec:Sahlqvist}

   Two central theorems in classical modal logic bear the name ``Sahlqvist":
   a \emph{correspondence} and a \emph{completeness} theorem~\cite[Chapters~3 and~5]{BRV01}.
   The first establishes a correspondence between a class of formulas,
   called ``Sahlqvist formulas", and first-order properties on frames:
   a Sahlqvist formula is valid on a frame if and only if its corresponding
   first-order frame property holds on the frame~\cite{Sah75}.
   The second furthers this result by showing that any extension of the classical
   modal logic $\log{K}$ with a set $\Lambda$ of Sahlqvist formulas is complete
   with respect to the class of frames described by the first-order correspondents
   of formulas in $\Lambda$.
   Sahlqvist completeness obviously relies on Sahlqvist correspondence, 
   but it also make essential use of a duality between descriptive (classical)
   frames and Boolean algebras with operators~\cite{SamVac89}.
   
  In this section, we provide Sahlqvist style results:
  we define a class of Sahlqvist formulas,
  establish a correspondence result for these,
  and then demonstrate completeness for extensions of $\log{CK}$
  with Sahlqvist formulas leveraging the duality established in
  Section~\ref{sec:duality}.

%=============================================================================%
\subsection{The Standard Translation}

We take our first steps towards building frame correspondents by introducing
the standard translation of intuitionistic modal formulas into a first-order
language.

\begin{definition}
  We take a classical first-order correspondence language $\FOL$
  with equality $=$,
  and predicate symbols $\leq$ and $R$ (corresponding to the intuitionistic
  and modal accessibility relation of \CK-frames), a predicate $P$ for each
  proposition letter $p \in \Prop$, and a constant symbol $\false$.
  Given a variable $x$, we recursively define the \emph{standard translation} of an
  intuitionistic modal formula on its structure as follows:
  \begin{align*}
    \st_x(p) &= Px \\
    \st_x(\bot) &= (x=\false) \\
    \st_x(\psi \wedge \chi) &= \st_x(\psi) \wedge \st_x(\chi) \\
    \st_x(\psi \vee \chi) &= \st_x(\psi) \vee \st_x(\chi) \\
    \st_x(\psi \to \chi) &= \forall y(x \leq y \wedge \st_y(\psi) \to \st_y(\chi)) \\
    \st_x(\Box\phi) &= \forall y(x \leq y \to \forall z(y R z \to \st_z(\phi))) \\
    \st_x(\Diamond\phi) &= \forall y(x \leq y \to \exists z(y R z \wedge \st_z(\phi)))
  \end{align*}
\end{definition}

  $\FOL$ is a convenient language to talk about models.
  In fact, every \CK-model $\mo{M}$ yields a structure $\mo{M}^{\circ}$
  interpreting $\FOL$: the intuitionistic and modal relations of $\mo{M}$ 
  interpret the corresponding binary predicates in $\FOL$,
  the valuation interprets the unary predicates,
  and $\expl$ interprets $\false$.

  The standard translation exactly shows what a formula expresses
  about the model satisfying it.
  More formally, a straightforward induction on the structure of $\phi$ gives:

\begin{proposition}
  Let $\mo{M}$ be a \CK-model.
  Then for any formula $\phi$ and world $w$ in $\mo{M}$ we have
  \begin{equation*}
    \mo{M}, w \Vdash \phi
      \iff \mo{M}^{\circ} \models \st_x(\phi)[x \mapsto w]
  \end{equation*}
%  and
  \begin{equation*}
    \mo{M} \Vdash \phi
      \iff \mo{M}^{\circ} \models \forall x(\st_x(\phi)).
  \end{equation*}
\end{proposition}

  To obtain a similar result for frames, rather than models, we make use
  of a \emph{second}-order translation. The target second-order language
  of this translation, called $\SOL$, is $\FOL$ but where unary predicates
  are now second-order variables, making $\SOL$ a \emph{monadic}
  second-order language.
  In $\SOL$ we can quantify over the unary predicates, thereby mimicking valuations.
  Since this second-order quantification
  interprets predicates as subsets of the model and we are only interested
  in certain subsets, namely~\emph{upsets}, we extend the second-order translation with formulas
  ensuring that only correct interpretations of unary predicates are picked.
  More precisely, for an unary predicate $P$ we define
  \begin{equation*}
    \isup(P) := P\false \wedge \forall x \forall y((x \leq y) \wedge Px \to Py).
  \end{equation*}
  When true of $P$, this formula ensures that the interpretation of this predicate
  is an upset containing the exploding world $\expl$.

\begin{definition}
  Let $\phi$ be formulas all of whose proposition letters
  are among $p_1, \ldots, p_n$, and let $P_1, \ldots, P_n$ be the
  corresponding unary predicates.
  Then the \emph{second-order translation} of $\phi$ is defined as
  \begin{equation*}
	\so(\phi) := \forall P_1 \cdots \forall P_n \forall x(\isup(P_1) \wedge \cdots \wedge \isup(P_n) \to \st_x(\phi)).
  \end{equation*}
\end{definition}

   Note that in $\so(\phi)$ we quantify over $x$, indicating that we are looking
   for \emph{global} frame correspondents.
  This globality is witnessed in the following lemma,
  where $\mo{X}^{\circ}$ is the second-order structure interpreting $\SOL$
  emerging from the \CK-frame $\mo{X}$.

\begin{proposition}\label{prop:so-trans}
  For an \CK-frame $\mo{X}$ and formula $\phi \in \Formulas$,
  we have $\mo{X} \Vdash \phi$ iff $\mo{X}^{\circ} \models \so(\phi)$.
\end{proposition}

  We are now in possession of (global) frame correspondents
  for all formulas. However, these are expressed
  in $\SOL$ as they contain second-order quantifiers.
  Next, we turn to the objective of showing that for a certain class of
  formulas the second-order correspondent can be algorithmically reduced to an
  equivalent first-order correspondent, via the elimination of its second-order quantifiers.

%=============================================================================%
\subsection{Sahlqvist Correspondence}\label{subsec:Sahlqvist}

We introduce a class of syntactically defined formulas, called \emph{Sahlqvist formulas},
whose structure allows us to eliminate the second-order quantifiers from their second-order translation,
while preserving equivalence.

\begin{definition}\label{def:Sahl} \
  \begin{enumerate}
    \item A formula $\phi \in \Form$ is called \emph{positive} if it
          is constructed from
          proposition letters, $\top$, $\bot$, $\wedge$, $\vee$, $\Box$ and $\Diamond$.
    \item A \emph{boxed atom} is a formula of the form
          $\Box^n p = \underbrace{\Box \cdots \Box}_{\text{$n$ times}} p$, where $p \in \Prop$.
          We view $p = \Box^0p$ as a boxed atom as well.
    \item A \emph{Sahlqvist antecedent} is a formula constructed from boxed
          atoms using $\top$, $\bot$, $\wedge$ and $\vee$.
    \item A \emph{Sahlqvist formula} is an implication 
          $\phi \to \psi$ where $\phi$ is a Sahlqvist antecedent
          and $\psi$ is a positive formula.
  \end{enumerate}
\end{definition}

\begin{example}
  Some examples of Sahlqvist formulas are
  $\Box p \to p$, $p \to \Diamond p$, $p \to \Box p$, $\Box p \to \Box\Box p$,
  $\Box p \to \Diamond p$, $p \to \Box\Diamond p$ and $\Box\Box p \to \Box p$.
\end{example}

The use of positivity in Sahlqvist formulas is motivated by
the fact that it preserves truth when expanding a valuation.

\begin{definition}
A formula $\phi$ is \emph{upward monotone} if
for any model $(\mo{X},V)$ and world $x \in X$ 
and valuation $V' : \Prop \to \upp(\mo{X})$ such that $V(p) \subseteq V'(p)$ for all $p$,
the following holds:
  \begin{equation*}
    (\mo{X},V),x \Vdash \phi
    \quad\text{implies}\quad
    (\mo{X},V'),x \Vdash \phi.
  \end{equation*}
\end{definition}

\begin{lemma}\label{lem:polmonot}
  If $\phi$ is positive then it is upward monotone.
\end{lemma}
\begin{proof}
  This follows from a routine induction on the structure of $\phi$.
\end{proof}

We finally establish a correspondence theorem in the style of Sahlqvist~\cite{Sah75},
between formulas and first-order properties.
As alluded to earlier, we obtain this result by eliminating second-order
quantifiers, by notably leveraging the upward monotonicity of the 
(positive) consequent of any Sahlqvist formula.

\begin{theorem}\label{thm:Sahl}
Every Sahlqvist formula has a computable, first-order, global correspondent on frames.
\end{theorem}
\begin{proof}
  By Proposition~\ref{prop:so-trans} we have $\mo{X} \Vdash \psi \to \chi$
  if and only if $\mo{X}^{\circ} \models \so(\psi \to \chi)$.
  We assume that no two quantifiers bind the same variable.
  Let $p_1, \ldots, p_n$ be the propositional variables occurring in
  $\psi$, and write $P_1, \ldots, P_n$ for their corresponding unary predicates.
  We assume that every proposition letter that occurs in $\chi$
  also occurs is $\psi$, for otherwise we may replace it by $\bot$ to
  obtain a formula which is equivalent in terms of validity on frames.

  \bigskip\noindent
  {\it Step 1.}
    By unfolding $\so(\psi \to \chi)$ we get the following formula, where $\ISUP$ is the conjunction
    of all the $\isup(P_i)$ for $1\leq i\leq n$.
     $$\forall P_1 \cdots \forall P_n\forall x(\ISUP \to \forall y(x\leq y \to \st_y(\psi) \to \st_y(\chi)))$$
     Given that $\forall x \forall y(x\leq y \to \st_y(\psi) \to \st_y(\chi))$ is equivalent to
     $\forall x(\st_x(\psi) \to \st_x(\chi))$, we can simplify the SO-translation of our implication to
     $$\forall P_1 \cdots \forall P_n\forall x(\ISUP \wedge \st_x(\psi) \to\POS)$$
     where we named $\POS$ the positive formula $\st_x(\chi)$.
  
  \bigskip\noindent
  {\it Step 2.}
    Using distributivity, we can write $\psi$ as a disjunction
    $\psi_1 \vee \cdots \vee \psi_k$ such that each of the $\psi_i$ is a
    conjunction of boxed atoms. Next, we rewrite
    \begin{align*}
      \forall P_1 \cdots \forall P_n &\forall x (\ISUP \wedge \st_x(\psi) \to \POS) \\
        &= \forall P_1 \cdots \forall P_n \forall x \Big(\Big(\bigvee_{i=1}^k \ISUP \wedge \st_x(\psi_i)\Big) \to \POS \Big) \\
        &= \forall P_1 \cdots \forall P_n \forall x \Big(\bigwedge_{i=1}^k (\ISUP \wedge \st_x(\psi_i) \to \POS) \Big) \\
        &= \bigwedge_{i=1}^k \big( \forall P_1 \cdots \forall P_n \forall x (\ISUP \wedge \st_x(\psi_i) \to \POS) \big)
    \end{align*}
    
    It suffices to find a correspondent for each of the conjuncts,
    as the conjunction of these will then be a correspondent of $\psi \to \chi$.
    So, without loss of generality, we assume that $\psi$ is a conjunction of
    boxed atoms, so that $\st_x(\psi)$ is a conjunction of formulas of the form
    $\forall y(xR_{\Box}^my \to P_iy)$, where $xR_{\Box}^my$ is a notation for

    $$\exists z_1\dots\exists z_{2m}(x\leq z_1 \wedge z_1Rz_2\wedge \dots \wedge z_{2m-2}\leq z_{2m-1} \wedge z_{2m-1}Rz_{2m}\land z_{2m}\leq y)$$
    where all $z_i$ are different and different from $x$ and $y$. Observe the critical case $xR_{\Box}^0y=x\leq y$.
    So we can write $\so(\psi \to \chi)$ as
    \begin{equation}\label{eq:Sahl-1}
      \forall P_1 \cdots \forall P_n \forall x (\ISUP \wedge \BOXAT \to \POS),
    \end{equation}
    where $\BOXAT$ is a conjunction of boxed atoms.
    
    \bigskip\noindent
    {\it Step 3.}
    Now we read off the minimal instances of the $P_i$ making the antecedent
    true. Let $P_i$ be a unary predicate and let
    $\forall y_1(xR_{\Box}^{\ell_1}y_1 \to P_iy_1), \ldots, 
     \forall y_t(xR_{\Box}^{\ell_t}y_t \to P_iy_t)$
    be the boxed atoms in $\psi$ containing $P_i$.
    Then we define
     \begin{equation*}
      \sigma(P_i)
        := \lambda u. ((u \neq \false) \to
           ((xR_{\Box}^{\ell_1}u) \vee \cdots \vee (xR_{\Box}^{\ell_t}u))).
    \end{equation*}
    Then for each second-order model $\mo{M}$ we have
    \begin{equation}\label{eq:Sahl-obs}
      \mo{M} \models (\ISUP\wedge\BOXAT)[w]
      \quad\text{implies}\quad
      \mo{M} \models \forall y (\sigma(P_i)y \to P_iy)[w].
    \end{equation}
    Note that the presence of $\mo{M}\models\ISUP[w]$ above is crucial as it
    makes sure that $\false$, an element satisfying $\sigma(P_i)$, is in the interpretation of $P_i$.
    
    \bigskip\noindent
  {\it Step 4.}
    We now use the $\sigma(P_i)$ as instantiations of the $P_i$.
    Let
    \begin{equation*}
      [\sigma(P_1)/P_1, \ldots, \sigma(P_n)/P_n] \forall x(\ISUP \wedge \BOXAT \to \POS)
    \end{equation*}
    be the formula arising from removing the second order quantifiers
    in~\eqref{eq:Sahl-1} and replacing each $P_i$ with $\sigma(P_i)$.
    By construction,
    $[\sigma(P_1)/P_1, \ldots, \sigma(P_n)/P_n](\ISUP \wedge \BOXAT)$ is true: 
    while it is obvious for $\BOXAT$, some care is required for $\ISUP$.
    First observe that the implication $\sigma(P_i)\false$ is trivially true,
    as $\false\neq\false$ is a contradiction.
    Second, we can establish that $\forall y \forall z((y \leq z) \wedge \sigma(P_i)y \to \sigma(P_i)z)$
    by noting that if $y=\false$ then $z=\false$
    and that  if $xR_{\Box}^{\ell}y$ then $y\leq z$ gives us
    $xR_{\Box}^{\ell}z$.
    With this established, we simplify our formula to the equivalent
    \begin{equation}\label{eq:Sahl-2}
      \forall x([\sigma(P_1)/P_1, \ldots, \sigma(P_n)/P_n]\POS).
    \end{equation}
    Since every proposition letter in $\chi$ is assumed to be in $\psi$,
    the formula $\forall x([\sigma(P_1)/P_1, \ldots, \sigma(P_n)/P_n]\POS)$
    contains no unary predicates. Hence it is a first order frame condition.
    
    We complete the proof by showing that for any \CK-frame $\mo{X}$,
    we have $\mo{X}^{\circ}$ satisfies the formula from~\eqref{eq:Sahl-1}
    (which is equivalent to $\so(\psi \to \chi)$) if and only if it satisfies
    the formula from~\eqref{eq:Sahl-2}.
    The implication from~\eqref{eq:Sahl-1} to~\eqref{eq:Sahl-2} is an instantiation
    of the second order quantifiers. For the other implication, assume
    \begin{equation*}
      \mo{X}^{\circ} \models \forall x([\sigma(P_1)/P_1, \ldots, \sigma(P_n)/P_n]\POS)
    \end{equation*}
    and let $\mo{M}$ a second-order model on $\mo{X}^\circ$ interpreting $P_1, \ldots, P_n$
    and $w$ an interpretation of $x$
    such that
    \begin{equation*}
      \mo{M} \models (\ISUP \wedge \BOXAT)[w]
    \end{equation*}
    We need to show that $\mo{M} \models \POS[w]$.
    From the above and~\eqref{eq:Sahl-obs}, we get for every $P_i$ that
    its interpretation in $\mo{M}$ is an extension of $\sigma(P_i)$, i.e.~
    \begin{equation}\label{eq:Sahl-ext}
      \mo{M} \models \forall y (\sigma(P_i)y \to P_iy)[w]
    \end{equation}
    Additionally, our initial assumption gives us
    \begin{equation*}
      \mo{M} \models [\sigma(P_1)/P_1, \ldots, \sigma(P_n)/P_n]\POS[w]
    \end{equation*}
    Now, $\mo{M} \models\ISUP[w]$ ensures us that the interpretations
    of $P_1,\ldots,P_n$ are indeed upsets containing $\false$.
    Consequently, the above combined with the positivity of $\POS$ and \eqref{eq:Sahl-ext}
    give us $\mo{M} \models \POS[w]$ via an application of Lemma~\ref{lem:polmonot}.
\end{proof}

We put the algorithm within our result to use by considering some examples of
interest to the modal logic community.

\begin{example}
  Consider $\Box p \to p$.
  The second order translation, once simplified via step 1, is
  \begin{align*}
    \so(\Box p \to p) &
      = \forall P \forall x(\isup(P) \wedge \st_x(\Box p) \to \st_x(p)) \\
     &  = \forall P \forall x( \isup(P) \wedge \underbrace{\forall y( x \leq y \to \forall z(y R z \to Pz))}_{\BOXAT} \to Px).
  \end{align*}
  The boxed atom of $p$ gives rise to
  $\sigma(P) = \lambda u. ((u \neq \false)\to\exists v \exists w(x \leq v \wedge vRw \wedge w \leq u))$.
  Substituting this for $P$ in the second order translation makes the antecedent
  true, and leaves us with
  \begin{equation*}
    \forall x (\sigma(P)x)
      = \forall x((x \neq \false) \to \exists v \exists w (x \leq v \wedge vRw \wedge w \leq x)).
  \end{equation*}
  Thus, in a diagram, validity of $\Box p \to p$ corresponds to the following
  being true for all $x$:
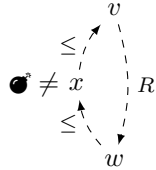
  \begin{equation*}
    \begin{tikzpicture}
      %% nodes
        \node (x) at (0,0) {$x$};
        \node     at (-.55,0) {$\expl \neq$};
        \node (v) at (.5,1) {$v$};
        \node (w) at (.5,-1) {$w$};
      %% edges
        \draw[-latex, bend left=15, dashed] (x) \ito{left} (v);
        \draw[-latex, bend left=15, dashed] (v) \rto{right} (w);
        \draw[-latex, bend left=15, dashed] (w) \ito{left} (x);
    \end{tikzpicture}
  \end{equation*}
  In this case, the antecedent of the correspondent also holds if $x = \false$,
  because we have $\expl \leq \expl R \expl \leq \expl$, so we can remove the
  condition $x \neq \expl$ from the diagram:
  
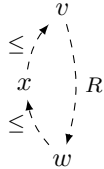
\begin{equation*}
    \begin{tikzpicture}
      %% nodes
        \node (x) at (0,0) {$x$};
        \node (v) at (.5,1) {$v$};
        \node (w) at (.5,-1) {$w$};
      %% edges
        \draw[-latex, bend left=15, dashed] (x) \ito{left} (v);
        \draw[-latex, bend left=15, dashed] (v) \rto{right} (w);
        \draw[-latex, bend left=15, dashed] (w) \ito{left} (x);
    \end{tikzpicture}
  \end{equation*}
\end{example}

\begin{example}
  Consider $p \to \Diamond p$. The simplified second order translation is
  \begin{align*}
    \so(p \to \Diamond p)
      &= \forall P \forall x (\isup(P) \wedge \st_x(p) \to \st_x(\Diamond p)) \\
      &= \forall P \forall x (\isup(P) \wedge Px \to \forall y(x \leq y \to \exists z(yRz \wedge Pz))).
  \end{align*}
  The atom $p$ in the antecedent gives $\sigma(P) = \lambda u.(u \neq \false) \to (x \leq u)$.
  Substituting this for $P$ in the second order translation makes the antecedent
  true, and yields the first-order correspondent
  \begin{equation*}
    \forall x \forall y(x \leq y \to \exists z(yRz \wedge (z \neq \false) \to (x \leq z))).
  \end{equation*}
  This is equivalent to the slightly simpler frame condition
  $\forall x \exists z (xRz \wedge ((z \neq \false) \to (x \leq z)))$.
\end{example}

\begin{example}
  Consider $\Box p \to \Diamond p$. Then
  \begin{equation*}
    \so(\Box p \to \Diamond p)
      = \forall P \forall x( \isup(P) \wedge \forall y( x \leq y \to \forall z(y R z \to Pz)) 
        \to \forall s (x \leq s \to \exists t (sRt \wedge Pt))).
  \end{equation*}
  We have the same definition of $\sigma(P)$ as above. Filling this in gives:
  \begin{equation*}
    \forall x\forall s (x \leq s \to \exists t(sRt \wedge ((t \neq \false) \to \exists v \exists w (x \leq v R w \leq t)))).
  \end{equation*}
  In a diagram:
  $$
    
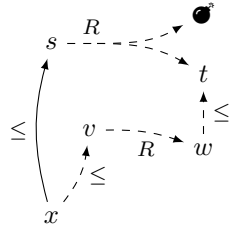
\begin{tikzpicture}
      %% nodes
        \node (x) at (0,.1) {$x$};
        \node (s) at (0,2.4) {$s$};
        \node (v) at (.5,1.25) {$v$};
        \node (w) at (2,1) {$w$};
        \node (t) at (2,2) {$t$};
        \node (e) at (2,2.8) {$\expl$};
      %% edges
        \draw[-latex, bend left=15] (x) \ito{left} (s);
        \draw[-latex, bend right=15, dashed] (x) \ito{right} (v);
        \draw[-latex, dashed, bend left=10] (v) \rto{below} (w);
        \draw[-, dashed] (s) \rto{above,pos=.8} (.6,2.4);
        \draw[-latex, dashed] (.6,2.4) to[out=0,in=150] (t);
        \draw[-latex, dashed] (.6,2.4) to[out=0,in=210] (e);
        \draw[-latex, dashed] (w) \ito{right} (t);
    \end{tikzpicture}
  $$
  We now see that we can simplify this by identifying $v = s$ and $w = t$
  to get a simpler condition. Furthermore, the universal quantification over
  $x$ implies that we may take $s = x$. So we get:
  $$
    \forall x \exists t (x R t).
  $$
  In other words, we simply find seriality of $R$.
\end{example}

\begin{example}
  In this example we consider the transitivity axiom $\Box p \to \Box\Box p$.
  In the constructive case, this no longer exactly corresponds to transitivity.
  So let us compute its correspondent.
  We have
  \begin{equation*}
    \so(\Box p \to \Box\Box p)
      = \forall P \forall x(\isup(P) \wedge \forall y \forall z(x \leq y Rz \to Pz))
        \to \forall s_1 \forall s_2 \forall s_3 \forall s_4(x \leq s_1 R s_2 \leq s_3 R s_4 \to Ps_4)
  \end{equation*}
  This gives $\sigma(P) = \lambda u. (u \neq \false) \to (\exists v \exists w(x \leq v \wedge v R w \wedge w \leq u))$,
  so that a first-order correspondent is given by
  \begin{equation*}
    \forall x \forall s_1 \forall s_2 \forall s_3 \forall s_4(x \leq s_1 R s_2 \leq s_3 R s_4 \neq \false \to \exists v \exists w (x \leq v R w \leq s_4)).
  \end{equation*}
  In a picture:
  \begin{equation*}
    \begin{tikzpicture}[xscale=.8]
      %% nodes
        \node (x)  at (0,-.5) {$x$};
        \node (s1) at (0,1) {$s_1$};
        \node (s2) at (2,1) {$s_2$};
        \node (s3) at (2,2) {$s_3$};
        \node (s4) at (4,2) {$s_4$};
        \node (e)  at (4.7,2) {$\neq \expl$};
        \node (v)  at (.75,.5) {$v$};
        \node (w)  at (3.75,.5) {$w$};
      %% edges
        \draw[-latex, bend left=10] (x) \ito{left} (s1);
        \draw[-latex, bend left=5]  (s1) \rto{above} (s2);
        \draw[-latex] (s2) \ito{left} (s3);
        \draw[-latex] (s3) \rto{above} (s4);
        \draw[-latex, dashed, bend right=10] (x) \ito{right} (v);
        \draw[-latex, dashed, bend right=5] (v) \rto{below} (w);
        \draw[-latex, dashed, bend right=5] (w) \ito{right} (s4);
    \end{tikzpicture}
  \end{equation*}
  This can be slightly simplified by equating $x$ and $s_1$, resulting
  in the following diagram:
  
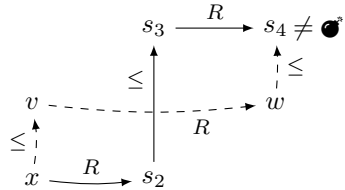
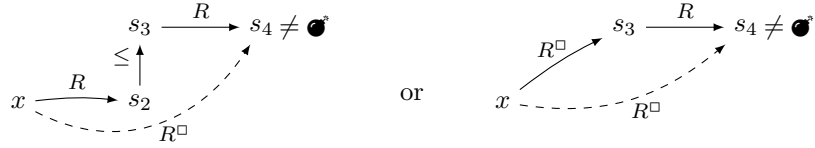
\begin{equation*}
    \begin{tikzpicture}[xscale=.8]
      %% nodes
        \node (x)  at (0,0) {$x$};
        \node (s2) at (2,0) {$s_2$};
        \node (s3) at (2,2) {$s_3$};
        \node (s4) at (4,2) {$s_4$};
        \node (e)  at (4.7,2) {$\neq \expl$};
        \node (v)  at (0,1) {$v$};
        \node (w)  at (4,1) {$w$};
      %% edges
        \draw[-latex, bend right=5]  (x) \rto{above} (s2);
        \draw[-latex] (s2) \ito{left,pos=.7} (s3);
        \draw[-latex] (s3) \rto{above} (s4);
        \draw[-latex, dashed, bend right=10] (x) \ito{left} (v);
        \draw[-latex, dashed, bend right=5] (v) \rto{below,pos=.75} (w);
        \draw[-latex, dashed, bend right=5] (w) \ito{right} (s4);
    \end{tikzpicture}
  \end{equation*}
  Finally, writing $R^{\Box} = ({\leq} \circ R \circ {\leq})$ yields:
  \begin{center}
    \begin{tikzpicture}[xscale=.8,baseline=1]
      %% nodes
        \node (x)  at (0,0) {$x$};
        \node (s2) at (2,0) {$s_2$};
        \node (s3) at (2,1) {$s_3$};
        \node (s4) at (4,1) {$s_4$};
        \node (e)  at (4.7,1) {$\neq \expl$};
      %% edges
        \draw[-latex, bend left=5]  (x) \rto{above} (s2);
        \draw[-latex] (s2) \ito{left,pos=.7} (s3);
        \draw[-latex] (s3) \rto{above} (s4);
        \draw[-latex, dashed, bend right=37] (x) to node[below,pos=.6]{\footnotesize{$R^{\Box}$}} (s4);
    \end{tikzpicture}
    \qquad or \qquad
    \begin{tikzpicture}[xscale=.8,baseline=1]
      %% nodes
        \node (x)  at (0,0) {$x$};
        \node (s3) at (2,1) {$s_3$};
        \node (s4) at (4,1) {$s_4$};
        \node (e)  at (4.7,1) {$\neq \expl$};
      %% edges
        \draw[-latex, bend left=5]  (x) to node[above,pos=.4]{\footnotesize{$R^{\Box}$}} (s3);
        \draw[-latex] (s3) \rto{above} (s4);
        \draw[-latex, dashed, bend right=20] (x) to node[below,pos=.6]{\footnotesize{$R^{\Box}$}} (s4);
    \end{tikzpicture}
  \end{center}
\end{example}

%===============================================================================
\subsection{Sahlqvist Completeness}\label{subsec:Sahl-compl}

  We leverage the correspondence just established to show that adding a set $\Lambda$ of Sahlqvist 
  formulas as axioms to \CK~leads to a logic which is complete with respect to the
  class of frames $\Lambda$ corresponds to.
  
  In the classical case, this is proved by showing that Sahlqvist formulas are
  \emph{d-persistent}: their validity on a descriptive frame entails their validity
  on the underlying frame.
  By algebraic completeness the formulas in $\Lambda$ are always evaluated to $\top$
  in the Lindenbaum algebra $\LT{\Lambda}$,
  hence valid in the dual descriptive frame $\LT{\Lambda}_*$.
  D-persistence then entails that $\Lambda$ is valid on the underlying
  (non-descriptive) frame $\LT{\Lambda}_+$.
  By Sahlqvist correspondence, this ensures that the ``canonical model" considered
  in the completeness proof, i.e.~$\LT{\Lambda}_+$, is indeed in the adequate class of frames.
  
  Unfortunately, the simplicity of this argument cannot be replicated here.
  The dual of the Lindenbaum algebra is to blame: 
  it is very ``noisy", in the sense that the $\leq$-clusters contain too wide a variety of segments.
  This prevents minimal valuations on this descriptive frame from being 
  \emph{closed}, a crucial component in the Sahlqvist completeness proof.
  We resolve this by introducing an intermediate step before 
  forgetting the admissible structure: we massage the descriptive frame 
  into a \emph{semi-descriptive} frame by pruning $\leq$-clusters.
  This procedure ensures that the minimal valuation is closed,
  allowing us to prove Sahlqvist completeness.

\begin{definition}\label{def:semi-descr}
  A \emph{semi-descriptive \CK-frame} is a general \CK-frame
  $\mo{G} = (X, \expl, \leq, R, A)$ that satisfies:
  \begin{enumerate}
    \item[\eqref{it:D1}]
          if $x \not\leq y$ then there exists $a \in A$ such that
          $x \in a$ and $y \notin a$, for all $x, y \in X$;
    \myitem{D2$'$}\label{it:D2p}
          for each $x \in X$,
          \begin{equation*}
            R[x] = \bigcap \{ a \in A \mid R[x] \subseteq a \} 
                   \cap \bigcap \{ -b \in -A \mid R[x] \subseteq -b \};
          \end{equation*}
    \item[\eqref{it:D3}]
          if $B \subseteq A \cup -A$ has the finite intersection property,
          then $\bigcap B \neq \emptyset$;
    \myitem{D4$'$}\label{it:D4p}
          for any $x \in X$ there exists $\widehat{x} \in X$ such that
          $x \sim \widehat{x}$ and $R[\widehat{x}] = \bigcap \{ a \in A \mid x \in \dbox a \}$.
  \end{enumerate}
\end{definition}

  Semi-descriptive frames have several useful properties,
  which allow for the Sahlqvist canonicity theorem to go through.
  These are most easily expressed using topological language.
  To this end, for any general \CK-frame $\mo{G} = (X, \expl, \leq, R, A)$
  let $\tau_A$ be the topology generated by the (clopen) subbases $A \cup -A$.
  Then $\mo{G}$ satisfies~\eqref{it:D3} if and only if $(X, \tau_A)$ is compact.
  Moreover, using the fact that $A$ is closed under binary unions and
  intersections, it can be seen that a set $C \subseteq X$ is \emph{closed}
  if and only if for each $x \in X \setminus C$ there exist $a \in A$
  and $-b \in -A$ such that $x \in a \cap -b$ and $a \cap -b$ is disjoint
  from $C$.

\begin{lemma}\label{lem:semdes-top}
  Let $\mo{G} = (X, \expl, \leq, R, A)$ be a semi-descriptive frame. Then
  \begin{enumerate}
    \item $\{ x \} = \bigcap \{ a \in A \mid x \in a \} 
                \cap \bigcap \{ -b \in -A \mid x \in -b \}$
          for every $x \in X$;
    \item if $C \subseteq X$ is closed, then so is
          ${\uparrow}_{\leq}C := \{ y \in X \mid x \leq y \text{ for some } x \in C \}$;
    \item if $C$ is a closed upset, then $R[C] := \bigcup \{ R[x] \mid x \in C \}$ is closed.
  \end{enumerate}
\end{lemma}
\begin{proof}
  The first item follows immediately from~\eqref{it:D1}.
  For the second item, suppose $C \subseteq X$ is closed and $y \notin {\uparrow}C$.
  Then $x \not\leq y$ for each $x \in C$, so we get some $a_x \in A$ containing
  $x$ but not $y$. Then
  \begin{equation*}
    C \subseteq \bigcup \{ a_x \mid x \in C \},
  \end{equation*}
  so by compactness we can find $a_1, \ldots, a_n$ such that
  $C \subseteq a_1 \cup \cdots \cup a_n$.
  Since each of $a_1, \ldots, a_n$ is an upset, we find
  ${\uparrow}C \subseteq a_1 \cup \cdots \cup a_n$,
  and hence $-a_1 \cap \cdots \cap -a_n$ is an open set containing $y$
  disjoint from ${\uparrow}C$, and we may conclude that ${\uparrow}C$ is closed.
  
  Finally, suppose $C$ is a closed upset.
  The for each $x \in C$ we have $\widehat{x} \in C$, because $C$ is an upset
  and $x \sim \widehat{x}$. By construction, $R[x] \subseteq R[\widehat{x}]$,
  so that
  \begin{equation*}
    R[C] = \bigcup \{ R[\widehat{x}] \mid x \in C \}.
  \end{equation*}
  Now suppose $y \notin R[C]$.
  Then for each $x \in C$ we have $y \notin R[\widehat{x}]$,
  so by definition of $\widehat{x}$ there exists some $a_x \in A$ such that
  $x \in \dbox a$ while $y \notin a_x$. This yields an open cover
  \begin{equation*}
    C \subseteq \bigcup \{ \dbox a_{\widehat{x}} \mid x \in C \}
  \end{equation*}
  and since $C$ is a closed, hence compact, we can find
  $a_1, \ldots, a_n$ such that $C \subseteq \dbox a_1 \cup \cdots \cup \dbox a_n$.
  As a consequence, $-a_1 \cap \cdots \cap -a_n$ is an open set
  containing $y$ disjoint from $R[C]$.
  This proves that $R[C]$ is closed.
\end{proof}

  Any descriptive \CK-frame can be turned into a semi-descriptive \CK-frame
  by pruning part of the worlds. This process is such that no entire cluster
  is omitted, and such that the remaining worlds satisfy precisely the same
  formulas in the pruned frame as they did in the original one.

\begin{definition}\label{def:ccc}
  Let $\mo{D} = (X, \expl, \leq, R, A)$ be a descriptive \CK-frame.
  For a world $x \in X$, let
  \begin{equation*}
    C_x = \bigcap \{ a \in A \mid R[x] \subseteq a \} 
          \cap \bigcap \{ -b \in -A \mid R[x] \cap b = \emptyset \}.
  \end{equation*}
  Then we can use~\eqref{it:D4} to find a (unique) world $\ccc{x} \in X$ such
  that $\ccc{x} \sim x$ and $R[\ccc{x}] = C_x$.
  We call $\ccc{x}$ the \emph{convex closed companion} of $x$.
\end{definition}
  
\begin{lemma}\label{lem:des-semdes}
  Let $\mo{D} = (X, \expl, \leq, R, A)$ be a descriptive frame.
  Define $\ccc{X} = \{ x \in X \mid \ccc{x} = x \}$ and let
  $\cleq$ and $\cR$ be the restrictions of $\leq$ and $R$ to $\ccc{X}$.
  For $a \in A$, define $\ccc{a} = a \cap \ccc{X}$ and let $\ccc{A} = \{ \ccc{a} \mid a \in A \}$.
  Then
  \begin{equation*}
    \ccc{\mo{D}} := (\ccc{X}, \expl, \cleq, \cR, \ccc{A})
  \end{equation*}
  is a semi-descriptive frame.
\end{lemma}
\begin{proof}
  Observe that $R[\expl] = \{ \expl \} \in A$, so $\ccc{\expl} = \expl \in \ccc{X}$.
  By definition of $\cleq$ the inconsistent world still satisfies the required
  maximality condition. We now prove that $\ccc{\mo{D}}$ is a general \CK-frame
  that satisfies the four conditions from Definition~\ref{def:semi-descr}.
  
  \medskip\noindent
  \textit{Claim 1: $\ccc{\mo{D}}$ is a general frame.} 
  The family $\ccc{A}$ contains $\{ \expl \}$ and $\ccc{X}$
  because $A$ contains $\{ \expl \}$ and $X$.
  A routine verification shows that $\ccc{A}$ is closed under binary intersections,
  and binary unions.
  Using the fact that for each $x \in X$ we have $x \sim \ccc{x} \in \ccc{X}$,
  it can easily be verified that $\ccc{a} \To \ccc{b} = \ccc{a \To b}$,
  so $\ccc{A}$ is also closed under $\To$.
  For the modal cases, we show $\ccc{\dbox a} = \dbox \ccc{a}$
  and $\ccc{\ddiamond a} = \ddiamond \ccc{a}$.

  Suppose $\ccc{x} \in \ccc{\dbox a}$. Then $\ccc{x} \in \dbox a \cap \ccc{X}$.
  So for all $y, z$ such that $\ccc{x} \leq y R z$ we have $z \in a$.
  In particular, if $y, z \in \ccc{X}$ then $\ccc{x} \ccc{\leq} y \ccc{R} z$,
   which implies $z \in a \cap \ccc{X}$, so $z \in \ccc{a}$.
  Therefore $\ccc{x} \in \dbox\ccc{a}$.
  For the converse, suppose $\ccc{x} \notin \ccc{\dbox a}$.
  Then $\ccc{x} \notin \dbox a$, so there exist $y, z \in X$ such that
  $\ccc{x} \leq y R z$ and $z \notin a$.
  Since $y \sim \ccc{y}$ and $z \sim \ccc{z}$,
  with the help of~\eqref{it:D2} we find
  $\ccc{x} \cleq \ccc{y} \ccc{R} \ccc{z}$.
  Furthermore, since $a$ is an upset we have $\ccc{z} \in a$, hence
  $\ccc{z} \in a \cap \ccc{X} = \ccc{a}$. Therefore $\ccc{x} \notin \dbox \ccc{a}$.
  
  Suppose $\ccc{x} \in \ccc{\ddiamond a}$. Then $\ccc{x} \in \ddiamond a \cap \ccc{X}$.
  So for all $y$ such that $\ccc{x} \leq y$ there is a $z$ with $y R z$ and $z \in a$.
  By letting $y \in \ccc{X}$ such that $\ccc{x} \ccc{\leq} y$, we obtain a $z\in X$ with $y R z$ and $z\in a$.
  Now, consider $\ccc{z}$: by $\ccc{z}\sim z$ any $a\in A$ (resp.~$-a\in -A$) satisfies
  $z\in a$ iff $\ccc{z}\in a$, given that it is an upset (resp.~downset).
  Following Definition~\ref{def:ccc} we get $y \ccc{R} \ccc{z}$ from $y R z$.
  Therefore $\ccc{x} \in \ddiamond\ccc{a}$.
  For the converse, suppose $\ccc{x} \notin \ccc{\ddiamond a}$.
  Then $\ccc{x} \notin \ddiamond a$, so there exists $y \in X$ with $\ccc{x} \leq y$
  such that for all $z\in X$ if $y R z$ then $z \notin a$.
  In particular, we obtain that for all $z\in \ccc{X}$ if $y \ccc{R} z$ then $z \notin \ccc{a}$.
  Again, we exploit $y \sim \ccc{y}$ to get $\ccc{x}\ccc{\leq}\ccc{y}$. 
  If we show that $z\in \ccc{X}$ if $\ccc{y} \ccc{R} z$ then $z \notin \ccc{a}$, then we are done.
  For such a $z$ we notably get $z\in  \bigcap \{ -b \in -A \mid R[y] \cap b = \emptyset \}$
  from Definition~\ref{def:ccc}.
  As $R[y]\cap a=\emptyset$, we get that $z\in -a$. 
  Set-theoretically speaking, $\ccc{a}\subseteq a$ informs us that $-a \subseteq -\ccc{a}$, hence $z\in -\ccc{a}$.
  As we established that $z\notin \ccc{a}$, we proved $\ccc{x} \notin \ddiamond \ccc{a}$.
  
  \medskip\noindent
  \textit{Claim 2: For any $x \in \ccc{X}$ and $a \in A \cup -A$, we have
    $R[x] \subseteq a$ iff $\cR[x] \subseteq \ccc{a}$.}
    The left-to-right inclusion follows immediately from the
    definitions of $\cR$ and $\ccc{a}$.
    For the converse, suppose $R[x] \not\subseteq a$.
    Then there exists some $y \in R[x]$ such that $y \notin a$.
    This implies that $\ccc{y} \in R[x]$ because $\mo{D}$ satisfies~\eqref{it:D2}
    and $y \sim \ccc{y}$, so $\ccc{y} \in \cR[x]$.
    Besides, $\ccc{y} \notin a$ because $a$ is an upset or a downset
    and $y \sim \ccc{y}$. Therefore $\cR[x] \not\subseteq \ccc{a}$.
  
  \medskip\noindent
  \textit{Claim 3: $\ccc{\mo{D}}$ satisfies~\eqref{it:D1}, \eqref{it:D2p}, \eqref{it:D3} and \eqref{it:D4p}.}
    The fact that $\mo{D}$ satisfies~\eqref{it:D1} implies that $\ccc{\mo{D}}$
    satisfies this as well. To see that $\ccc{\mo{D}}$ satisfies~\eqref{it:D2p},
    combine the definition of $\ccc{x}$, Claim 2, and the fact that
    $\ccc{-b} = -\ccc{b}$ for all $b \in A$, to find that for any
    $x \in \ccc{X}$,
    \begin{align*}
      \cR[x]
        &= \ccc{X} \cap R[x]
        &\text{(definition of $\cR$)} \\
        &= \ccc{X} \cap \bigcap \{ a \in A \mid R[x] \subseteq a \} 
                   \cap \bigcap \{ -b \in -A \mid R[x] \subseteq -b \}
        &\text{(because $x \in \ccc{X}$)} \\
        &= \ccc{X} \cap \bigcap \{ a \in A \mid \cR[x] \subseteq \ccc{a} \} 
                   \cap \bigcap \{ -b \in -A \mid \cR[x] \subseteq \ccc{-b} \}
        &\text{(by Claim 2)} \\
        &= \bigcap \{ a \cap \ccc{X} \in \ccc{A} \mid \cR[x] \subseteq \ccc{a} \} 
           \cap \bigcap \{ -b \cap \ccc{X} \in -\ccc{A} \mid \cR[x] \subseteq \ccc{-b} \} \\
        &= \bigcap \{ \ccc{a} \in \ccc{A} \mid \cR[x] \subseteq \ccc{a} \} 
           \cap \bigcap \{ -\ccc{b} \in -\ccc{A} \mid \cR[x] \subseteq -\ccc{b} \}
        &\text{(definition of $\ccc{a}$)}
    \end{align*}

    For~\eqref{it:D3},
    suppose $B \subseteq \ccc{A} \cup -\ccc{A}$ has the finite intersection property.
    Then clearly the set
    $B^{\circ} := \{ b \in A \cup -A \mid \ccc{b} \in B \} \subseteq A \cup -A$
    has the finite intersection property, so $\mo{D}$ satisfies~\eqref{it:D3}
    we find $\bigcap B^{\circ} \neq \emptyset$.
    Let $y \in \bigcap B^{\circ}$. Then $\ccc{y} \in \bigcap B$,
    so $\bigcap B \neq \emptyset$.
    Therefore $\ccc{\mo{D}}$ satisfies~\eqref{it:D3}.

    Finally, to prove that~\eqref{it:D4p} holds, suppose $x \in \ccc{X}$
    and let $U = \bigcap \{ a \in A \mid x \in \dbox a \} \subseteq X$.
    Then $R[x] \subseteq U$ by definition of $\dbox$.
    If $x \in \ddiamond a$ for some $a \in A$ then $R[x] \cap a \neq \emptyset$,
    hence $U \cap a \neq \emptyset$. 
    Using the fact that $\mo{D}$ satisfies~\eqref{it:D4}, we find some
    $\widehat{x} \in X$ such that $x \sim \widehat{x}$ and $R[\widehat{x}] = U$.
    Then clearly $\widehat{x} \in \ccc{X}$. Furthermore,
    \begin{align*}
      \cR[\widehat{x}]
        &= \ccc{X} \cap R[\widehat{x}]
        &\text{(definition of $\cR$)} \\
        &= \ccc{X} \cap \bigcap \{ a \in A \mid R[x] \subseteq a \}
        &\text{(definition of $\widehat{x}$)} \\
        &= \ccc{X} \cap \bigcap \{ a \in A \mid \cR[x] \subseteq \ccc{a} \}
        &\text{(by Claim 2)} \\
        &= \bigcap \{ a \cap \ccc{X} \mid \cR[x] \subseteq \ccc{a} \} \\
        &= \bigcap \{ \ccc{a} \mid \cR[x] \subseteq \ccc{a} \},
        &\text{(definition of $\ccc{a}$)}
    \end{align*}
    as desired.
\end{proof}

  Recall that we can extend the valuation $V$ of a (general) \CK-frame
  to all formulas by letting $V(\phi)$ be the truth set of $\phi$.

\begin{lemma}\label{lem:ccc-D}
  Let $\mo{D} = (X, \expl, \leq, R, A)$ be a descriptive frame
  and let $V$ be an admissible valuation for $\mo{D}$.
  Define the valuation $\ccc{V}$ for $\ccc{\mo{D}}$ by
  $\ccc{V}(p) = V(p) \cap \ccc{X}$.
  Then for any formula $\phi$ we have $\ccc{V}(\phi) = \ccc{V(\phi)}$.
\end{lemma}
\begin{proof}
    This follows from a straightforward induction on the structure of $\phi$.
    While the base case holds by definition, for the inductive cases 
    we use that $x \in a$ iff $\ccc{x} \in \ccc{a}$ for all $x \in X$ and $a \in A$.
    We give the cases of $\Diamond$.
    
   From $x\in\ccc{V}(\Diamond\phi)$ we get $x\in\ddiamond\ccc{V}(\phi)$ by definition. By induction hypothesis we 
   have $x\in\ddiamond\ccc{V(\phi)}$, hence $x\in\ccc{\ddiamond V(\phi)}$ via the equality 
   $\ddiamond\ccc{a}=\ccc{\ddiamond a}$ obtained as the one for $\ddiamond$ above.
   As all transformations we performed are equivalences, we established $\ccc{V}(\Diamond\phi)=\ccc{\ddiamond V(\phi)}$.
\end{proof}
  
  It follows from Lemma~\ref{lem:ccc-D} that a descriptive frame $\mo{D}$
  and its pruning $\ccc{\mo{D}}$ (in)validate precisely the same formulas.
  In what follows, we prove that validity of Sahlqvist formulas is preserved
  when moving from a semi-descriptive frame $\mo{G}$
  to its underlying \CK-frame $\kappa\mo{G}$.
  Combining these, we conclude the completeness of extensions of \CK with
  Sahlqvist formulas, with respect to the adequate class of frames:
  the validity of Sahlqvist formulas can travel all the way from the
  Lindenbaum-Tarski algebra $\LT{}$ to the frame $\kappa\ccc{\LT{}_{*}}$
  generated via duality and pruning.

\begin{definition}
  If $\mo{G} = (X, \expl, \leq, R, A)$ is a general \CK-frame,
  then we write $\kappa\mo{G} = (X, \expl, \leq, R)$ for the underlying
  \CK-frame.
  In particular, if $\mo{D}$ is a descriptive frame then
  $\kappa\overline{\mo{D}}$ denotes the \CK-frame underlying the pruned
  general frame $\overline{\mo{D}}$.
\end{definition}

\begin{definition}
  A formula $\phi$ is called \emph{pruning persistent} or \emph{p-persistent}
  if for every descriptive
  \CK-frame $\mo{D}$,
  \begin{equation*}
    \mo{D} \Vdash \phi
    \quad\text{implies}\quad
    \kappa\overline{\mo{D}} \Vdash \phi.
  \end{equation*}
\end{definition}

\begin{definition}
  Let $\mo{D} = (X, \expl, \leq, R, A)$ be a general \CK-frame.
  We call a (not-necessarily admissible)
  valuation $V$ of $\mo{D}$
  \emph{closed} if $V(p) = \bigcap \{ a \in A \mid V(p) \subseteq a \}$
  for all $p \in \Prop$.
  Furthermore, we write $V \lessdot U$ if $U$ is an admissible valuation
  for $\mo{D}$ such that $V(p) \subseteq U(p)$ for all $p \in \Prop$.
\end{definition}

\begin{lemma}\label{lem:closval}
  Let $\mo{G} = (X, \expl, \leq, R, A)$ be a semi-descriptive \CK-frame
  and let $V$ be a closed valuation of the proposition letters.
  Then for any positive formula $\phi$ we have
  \begin{equation*}
    V(\phi) = \bigcap_{V \lessdot U} U(\phi).
  \end{equation*}
\end{lemma}
\begin{proof}
  We use induction on the structure of $\phi$.
  If $\phi$ is a proposition letter then the lemma follows immediately from
  the definition of a closed valuation, and if $\phi = \bot$ then it follows
  from the fact that $\bot$ is always interpreted as $\{ \expl \}$.
  The cases for conjunction and disjunction proceed as in~\cite[Theorem~5.91]{BRV01}.
  We consider the remaining cases.

  \medskip\noindent
  \textit{Case $\phi = \Box\psi$.}
    In this case, we have
    \begin{equation*}
      V(\Box\psi)
        = \{ x \in X \mid R[{\uparrow}x] \subseteq V(\psi) \}
        = \bigcap_{V \lessdot U} \{ x \in X \mid R[{\uparrow}x] \subseteq U(\psi) \}
        = \bigcap_{V \lessdot U} U(\Box\psi).
    \end{equation*}

  \medskip\noindent
  \textit{Case $\phi = \Diamond\psi$.}
    It follows from the fact that $\psi$ is positive that $V \lessdot U$ implies
    $V(\Diamond\psi) \subseteq U(\Diamond\psi)$,
    so $V(\Diamond\psi) \subseteq \bigcap_{V \lessdot U} U(\Diamond\psi)$.
    For the converse inclusion, suppose $x \notin V(\Diamond\psi)$.
    Then there exists some $y \geq x$ such that $R[y] \cap V(\Diamond\psi) = \emptyset$.
    Using the induction hypothesis and the fact that $\mo{G}$ is semi-descriptive,
    we find
    \begin{equation*}
      \underbrace{%
        \bigcap \{ a \in A \mid R[y] \subseteq a \} 
        \cap \bigcap \{ -b \in -A \mid R[y] \subseteq -b \}}_{= R[y]}
      \cap
        \underbrace{\bigcap \{ U(\psi) \in A \mid V \lessdot U \}}_{=V(\psi)}
      = \emptyset.
    \end{equation*}
    Since $\ccc{\mo{G}}$ is compact (by Lemma~\ref{lem:des-semdes}),
    we can find a finite number $U_1, \ldots, U_n$ of admissible valuations
    such that $R[y] \cap U_1(\psi) \cap \cdots \cap U_n(\psi) = \emptyset$.
    Now define admissible valuation $U'$ by $U'(p) = U_1(p) \cap \cdots \cap U_n(p)$
    for all $p \in \Prop$. Then $V \lessdot U'$,
    and by construction $x \notin U'(\Diamond\psi)$.
    Therefore $x \notin \bigcap_{V \lessdot U} U(\Diamond\psi)$.
    This proves $\bigcap_{V \lessdot U} U(\Diamond\psi) \subseteq V(\psi)$.
\end{proof}

\begin{theorem}\label{thm:sahlcan}
  Every Sahlqvist formula is p-persistent.
\end{theorem}
\begin{proof}
  Let $\phi = \psi \to \chi$ be a Sahlqvist formula and $\mo{D}$ a descriptive
  \CK-frame such that $\mo{D} \Vdash \phi$.
  Then by Lemma~\ref{lem:ccc-D}, $\ccc{\mo{D}} \Vdash \phi$.
  We prove that $\kappa\ccc{\mo{D}} \Vdash \phi$ by proving that
  $\ccc{\mo{D}}$ satisfies the Sahlqvist correspondent of $\phi$, which,
  after steps 2 and 3 of the proof of Theorem~\ref{thm:Sahl}, 
  is equivalent to
  
  $$\ccc{\mo{D}}^\circ\models\forall P_1\dots\forall P_n\forall x(\ISUP\land\BOXAT\to\POS)$$
  
  \noindent Noting that $x$ does not appear free in $\ISUP$, to establish our result we need to pick arbitrary interpretations of all $P_i$
  satisfying $\isup(P_i)$.
  Additionally, the formula $\forall x(\ISUP\land\BOXAT\to\POS)$ is insensitive to the interpretation
  of second-order predicates not present in it.
  Consequently, without loss of generality we can restrict our attention to the interpretation of all predicates satisfying $\isup$,
  that is \emph{valuations}.
  This reduces our goal to

  $$\forall V.\;\;(\ccc{\mo{D}},V)^\circ\models\forall x(\BOXAT\to\POS)$$
  
  \noindent Let us therefore pick a world $w$ of $\ccc{\mo{D}}$, and consider the minimal valuation
  defined below.
  
  $$V_m(p):=\bigcup\{
  	\ccc{R}_\Box^{\ell}[w]   \mid
	\text{the formula }
	\forall y(w\ccc{R}_{\Box}^{\ell}y \to Py)	\text{ appears in }\BOXAT
	\}$$
          
  \noindent Next, we note the equivalence between 
  $\forall V.\;\;(\ccc{\mo{D}},V)^\circ\models\BOXAT\to\POS[w]$ and
  
  $$(\ccc{\mo{D}},V_m)^\circ\models\BOXAT\to\POS [w]$$
  
  \noindent One direction of the equivalence is immediate: instantiate $V$ by $V_m$.
  For the other direction, recall that $V_m$ is a minimal valuation for $\BOXAT[w]$, which
  tells us of any valuation satisfying this formula that it is an extension of $V_m$.
  Now observe that $(\ccc{\mo{D}},V_m)^\circ\models\BOXAT[w]$ holds
  by construction of $V_m$, which entails that $(\ccc{\mo{D}},V_m)^\circ\models\POS[w]$
  by assumption.
  It suffices to use together the positivity of $\POS[w]$, this last fact, and Lemma~\ref{lem:polmonot} to get
  $(\ccc{\mo{D}},V)^\circ\models\POS[w]$.
  
  With this equivalence and the observation $(\ccc{\mo{D}},V_m)^\circ\models\BOXAT[w]$,
  we reduce our goal to 

  $$\;\;(\ccc{\mo{D}},V_m)^\circ\models\POS [w]$$
  
  \noindent To establish this fact, we leverage an important property of our minimal valuation:
  $V_m$ is closed.
  By applying finitely many times the items of Lemma~\ref{lem:semdes-top}, we can reach this conclusion.
  Indeed, item 1 informs us that singleton set are closed,
  the upset of a singleton set is also closed by item 2, 
  and the union of the sets of modal successors of these upsets is also closed by item 3.
  By finitely repeating this process, we can show that $V_m$ is closed, given
  that it is characterised by finite paths in $\ccc{R}_\Box^{\ell}[w]$.
  
  As it is closed, we can apply Lemma~\ref{lem:closval} on $V_m$ to put our goal to its final form:
  
  $$\forall U.\;V_m\lessdot U\;\to\;(\ccc{\mo{D}},U)^\circ\models\POS [w]$$
  
  Let $U$ be a valuation such that $V_m\lessdot U$, making $U$ an \emph{admissible} valuation extending $V_m$.
  Therefore, we obtain $(\ccc{\mo{D}},U)^\circ\models\BOXAT[w]$.
  Given that $\BOXAT$ is the \emph{first-order} correspondent of $\psi$,
  we straightforwardly get $(\ccc{\mo{D}},U),w\Vdash\psi$.
  Our initial assumption  $\ccc{\mo{D}}\Vdash\psi\to\chi$ therefore leads to $(\ccc{\mo{D}},U),w\Vdash\chi$.
  We finish our proof by exploiting once more the first-order correspondence to get $(\ccc{\mo{D}},U)^\circ\Vdash\POS[w]$.
\end{proof}

\begin{theorem}
For any set $\Lambda$ of Sahlqvist formulas,
the logic $\log{CK}\oplus\Lambda$ is complete w.r.t.
the class of frames characterised by the first-order frame correspondents generated from  $\Lambda$.
\end{theorem}

\begin{proof}
Assume $\not\vdash_{\Lambda}\phi$ for $\phi\in\Formulas$.
The proof of Theorem~\ref{thm:pdt} informs us that $\top\not\leq\int{\val}_{\scriptscriptstyle\LT{\Lambda}}(\phi)$.
Therefore, we get that $\LT{\Lambda}_{*}\not\Vdash\phi$,
so Lemma~\ref{lem:ccc-D} yields $\ccc{\LT{\Lambda}_{*}}\not\Vdash\phi$.
Since every admissible valuation is in particular a valuation, this implies
$\kappa\ccc{\LT{\Lambda}_{*}}\not\Vdash\phi$.

Now, it suffices to show that $\kappa\ccc{\LT{\Lambda}_{*}}$
is in the adequate class of frames.
For any $\lambda\in\Lambda$ we have $\top=\int{\val}_{\scriptscriptstyle\LT{\Lambda}}(\lambda)$ and
hence  $\LT{\Lambda}_{*}\Vdash\lambda$.
Any such $\lambda$ is a Sahlqvist implication and therefore p-persistent by Theorem~\ref{thm:sahlcan}, 
giving us $\kappa\ccc{\LT{\Lambda}_{*}}\Vdash\lambda$.
By Theorem~\ref{thm:Sahl} the first-order correspondent of $\lambda$ holds of $\kappa\ccc{\LT{\Lambda}_{*}}$.
As $\lambda$ is arbitrary, we get that $\kappa\ccc{\LT{\Lambda}_{*}}$ is indeed in the
class of frames characterised by the first-order frame correspondents generated from  $\Lambda$,
allowing us to conclude completeness.
\end{proof}

%%%%%%%%%%%%%%%%%%%%%%%%%%%%%%%%%%%%%%%%%%%%%%%%%%%%%%%%%%%%%%%%%%%%%%%%%%%%%%%%
\section{Goldblatt-Thomason theorem}\label{Sec:GT}

  In this section we use the duality between descriptive \CK-frames and
  \CK-algebras to obtain a definability theorem akin to Goldblatt and Thomason's
  definability theorem from~\cite{GolTho74}.
  This gives sufficient conditions for a class of \CK-frames to be axiomatic.
  More precisely, it states that a class of \CK-frames that is closed under
  so-called \emph{segment extensions} is axiomatic if and only if it reflects
  segment extensions and is closed under disjoint unions, generated subframes
  and bounded morphic images.
  The proof is obtained by dualising Birkhoff's variety theorem,
  using the segment extensions as a bridge between frames and algebras.
  
  The Goldblatt-Thomason theorem for classical modal logics usually makes use
  of some form of \emph{ultrafilter extension} of a frame to bridge the gap
  between algebras and frames
  (see e.g.~\cite[Definition~2.57]{BRV01}, \cite[Definition~4.34]{Han03}
  and~\cite[Section~3.2]{KurRos07}).
  When working with positive modal logics or intuitionistic modal logics, this is
  often replaced by the \emph{prime filter extension} of a frame
  (such as in~\cite[Section~5]{CelJan01}, \cite[Section~6]{Gol05} and~\cite{Gro22gt}).
  In either case, this extension of a frame $\mo{X}$ is obtained as the double
  dual of a frame, that is, as the frame underlying the descriptive frame dual
  to the complex algebra of $\mo{X}$.
  When we employ the same method, we do not end up with a frame based on the
  collection of prime filters of $\mo{X}$, but rather it is based on the
  set of \emph{segments}. Accordingly, we define the \emph{segment extension}
  of a \CK-frame as follows.

\begin{definition}
  The \emph{segment extension} of a \CK-frame $\mo{X}$ is defined as
  \begin{equation*}
    \seg\mo{X} = (\mo{X}^+)_+.
  \end{equation*}
\end{definition}

  Concretely, $\seg\mo{X}$ consists of the collection of segments of the
  complex algebra $\mo{X}^+$ of $\mo{X}$,
  with $\expl^{\seg} = (\upp(\mo{X}), \{ \upp(\mo{X}) \})$
  and relations given by
  \begin{align*}
    (\ff{p}, \Gamma) \subsetsim (\ff{q}, \Delta) &\iff \ff{p} \subseteq \ff{q} \\
    (\ff{p}, \Gamma) R^{\seg} (\ff{q}, \Delta) &\iff \ff{q} \in \Gamma
  \end{align*}

\begin{definition}
  The \emph{segment extension} of a \CK-model $\mo{M}$ is given by
  \begin{equation*}
    \seg\mo{M} = (\seg\mo{X}, V^{\seg}),
  \end{equation*}
  where $V^{\seg}(q) = \{ (\ff{p}, \Gamma) \in X^{\seg} \mid V(q) \in \ff{p} \}$.
\end{definition}

  \CK-frames and -models are closely related to their segment extensions:

\begin{lemma}\label{lem:seg-ext-truth}
  Let $\mo{X} = (X, \expl, \leq, R)$ be a \CK-frame and
  $\mo{M} = (\mo{X}, V)$ a \CK-model.
    \begin{enumerate}
    \item For all segments $(\ff{p}, \Gamma) \in X^{\seg}$ we have
          $\seg\mo{M}, (\ff{p}, \Gamma) \Vdash \phi$ iff $V(\phi) \in \ff{p}$.
    \item For all worlds $x \in X$ we have
          $\mo{M}, x \Vdash \phi$ iff $\seg\mo{M}, \Eta(x) \Vdash \phi$.
    \item If $\seg\mo{X} \Vdash \phi$ then $\mo{X} \Vdash \phi$.
  \end{enumerate}
\end{lemma}
\begin{proof}
  Since $\Thet(V(\phi)) = \{ (\ff{p}, \Gamma) \in X^{\seg} \mid V(\phi) \in \ff{p} \}$, the first item is equivalent to
  \begin{equation*}
    V^{\seg}(\phi) = \Thet(V(\phi)).
  \end{equation*}
  We prove this by induction on the structure of $\phi$.
  If $\phi = \bot$ then
  \begin{equation*}
    V^{\seg}(\bot)
      = \{ \expl^{\seg} \}
      = \Thet( \{ \expl \})
      = \Thet( V(\bot)).
  \end{equation*}
  If $\phi \in \Prop$ then the result follows immediately from the definitions
  of $V^{\seg}$ and $\Thet$.
  The inductive cases for $\phi = \psi \wedge \chi$ and $\phi = \psi \vee \chi$
  are routine, and
  the cases for implication, box and diamond follow immediately from Lemma~\ref{lem:A}.
  For example, if $\phi = \Box\psi$ then we find
  \begin{align*}
    \Thet(V(\Box\psi))
      &= \Thet(\dbox V(\psi))
      &\text{(by definition of $\dbox$)} \\
      &= \dbox\Thet(V(\psi))
      &\text{(by Lemma~\ref{lem:A})} \\
      &= V^{\seg}(\Box\psi)
      &\text{(by definition of $\dbox$)}
  \end{align*}

  The second item then follows from the first and the definition of
  $\Eta$ via
  \begin{equation*}
    \mo{M}, x \Vdash \phi
      \iff x \in V(\phi)
      \iff V(\phi) \in \eta(x)
      \iff \seg\mo{M}, \Eta(x) \Vdash \phi.
  \end{equation*}
  For the third item, let $V$ be any valuation for $\mo{X}$ and $x \in X$.
  Then $V^{\seg}$ is a valuation of $\seg\mo{X}$ and by assumption
  $(\seg\mo{X}, V^{\seg}), \Eta(x) \Vdash \phi$, so by (2)
  we get $(\mo{X}, V), x \Vdash \phi$.
  Since $V$ and $x$ are arbitrary, this proves $\mo{X} \Vdash \phi$.
\end{proof}

\begin{lemma}\label{lem:alg-axiomatic-variety}
  $\ms{C} \subseteq \CKAlg$ is axiomatic
  if and only if it is a variety of algebras.
\end{lemma}
\begin{proof}
  If
  $\ms{C} = \{ \alg{A} \in \CKAlg \mid \alg{A} \models \Phi \}$, 
  then it is precisely the variety of algebras satisfying
  $\phi^x \leftrightarrow \top$, where $\phi \in \Phi$ and $\phi^x$ is
  the formula we get from $\phi$ by replacing the proposition letters with
  variables from some set $S$ of variables.
  Conversely, suppose $\ms{C}$ is a variety of algebras given by a set
  $E$ of equations using variables in $S$. For each equation $\phi = \psi$
  in $E$, let $(\phi \leftrightarrow \psi)^p$ be the formula we get from
  uniformly replacing the variables in $\phi \leftrightarrow \psi$ with proposition
  letters. Then we have
  $\mc{C} = \Alg \{ (\phi \leftrightarrow \psi)^p \mid \phi = \psi \in E \}$.
\end{proof}

  For a class $\ms{K}$ of \CK-frames, write
  $\ms{K}^+ = \{ \mo{X}^+ \mid \mo{X} \in \ms{K} \}$ for the collection
  of corresponding complex algebras.
  Also, if $\ms{C}$ is a class of algebras, then we write
  $H\ms{C}$, $S\ms{C}$ and $P\ms{C}$ for its closure under
  \textbf{h}omomorphic images, \textbf{s}ubalgebras and \textbf{p}roducts,
  respectively.

\begin{lemma}\label{lem:mdv}
  A class $\ms{K} \subseteq \CKFrm$ is axiomatic
  if and only if
  \begin{equation}\label{eq:mdv}
    \ms{K} = \{ \mo{X} \in \CKFrm
                \mid \mo{X}^+ \in HSP(\ms{K}^+) \}.
  \end{equation}
\end{lemma}
\begin{proof}
  Suppose $\ms{K}$ is axiomatic, for the set of formulas $\Phi$.
  Then it follows from Lemma~\ref{lem:frm-to-alg} and the fact that $H$,
  $S$ and $P$ preserve validity of formulas that
  \eqref{eq:mdv} holds.
  Conversely, suppose \eqref{eq:mdv} holds.
  Since $HSP(\ms{K}^+)$ is a variety, Birkhoff's variety theorem
  states that it is of the form $\Alg\Phi$.
  It follows that $\ms{K}$ is axiomatic for $\Phi$.
\end{proof}

  We now have all the ingredients to prove a $\CK$-analogue of the Goldblatt-Thomason theorem.

\begin{theorem}\label{thm:gt}
  Let $\ms{K} \subseteq \CKFrm$ be closed under 
  segment extensions.
  Then $\ms{K}$ is axiomatic if and only if $\ms{K}$ reflects 
  segment extensions and is closed under disjoint
  unions, generated subframes and bounded morphic images.
\end{theorem}
\begin{proof}
  The implication from left to right follows from the fact that axiomatic classes
  are closed under disjoint unions, generated subframes, bounded morphic images
  (Proposition~\ref{prop:ax-class-1}) and reflect segment extensions
  (Lemma~\ref{lem:seg-ext-truth}).

  For the converse, by Lemma~\ref{lem:mdv} it suffices to prove that
  $\ms{K} = \{ \mo{X} \in \CKFrm \mid \mo{X}^+ \in HSP(\ms{K}^+) \}$.
  So let $\mo{X} = (X, \exp, \leq, R) \in \CKFrm$ and
  suppose $\mo{X}^+ \in HSP(\ms{K}^+)$.
  Then there are $\mo{Z}_i \in \ms{K}$ such that $\mo{X}^+$ is the
  homomorphic image of a subalgebra $\alg{A}$ of the
  product of the $\mo{Z}_i^+$.
  In a diagram:
  \begin{center}
    \begin{tikzcd}
      \mo{X}^+
        & [2.5em]
          \alg{A}
            \arrow[l, ->>, "\text{surjective}" {above,pos=.46}]
            \arrow[r, >->, "\text{injective}"]
        & [2.5em]
          \prod \mo{Z}_i^+
    \end{tikzcd}
  \end{center}
  Since $\prod\mo{Z}_i^+ \cong (\coprod \mo{Z}_i)^+$ by Lemma~\ref{lem:prod-coprod},
  dually this yields
  \begin{center}
    \begin{tikzcd}
      (\mo{X}^+)_+
            \arrow[r, >->, "\text{gen.~subframe}"]
        & [5em]
          \alg{A}_+
        & [8em]
          \big(\big(\coprod \mo{Z}_i\big)^+\big)_+
            \arrow[l, ->>, "\text{bounded morphic image}" {above,pos=.46}]
    \end{tikzcd}
  \end{center}
  We have $\coprod \mo{Z}_i \in \ms{K}$ because $\ms{K}$ is closed under coproducts,
  and $\big((\coprod \mo{Z}_i)^+\big)_+ \in \ms{K}$ because $\ms{K}$ is
  closed under segment extensions.
  Then $\alg{A}_+ \in \ms{K}$ and $\seg\mo{X} = (\mo{X}^+)_+ \in \ms{K}$ because 
  $\ms{K}$ is closed under bounded morphic images and generated subframes.
  Finally, since $\ms{K}$ reflects segment extensions we find
  $\mo{X} \in \ms{K}$.
\end{proof}

%%%%%%%%%%%%%%%%%%%%%%%%%%%%%%%%%%%%%%%%%%%%%%%%%%%%%%%%%%%%%%%%%%%%%%%%%%%%%%%%
\section{Conclusion}

  We have continued the semantic study of the logic $\CK$ by providing a duality
  theorem and using this to obtain analogues of Sahlqvist results and of the
  Goldblatt-Thomason theorem. This paves the way for further
  investigation of the semantics of $\CK$ and related logics.
  In particular, we are interested in the following questions:

\begin{description}
  \item[Extending the class of Sahlqvist formulas]
        While the class of Sahlqvist formulas from Definition~\ref{def:Sahl}
        covers a wide variety of examples, it does not admit the use of
        the diamond operator in the antecedent.
        The difficulty in handling diamonds is caused by the $\forall\exists$-pattern
        of its interpretation. It would be interesting to investigate
        whether we can ease this restriction, allowing (some appearances of)
        diamonds in the antecedent.
  \item[Different dualities]
        As we have seen, descriptive frames for $\CK$ are necessarily based on
        preordered sets, rather than partially ordered ones.
        However, there appears to be a choice in how big we make our clusters,
        i.e.~which segments we admit.
        In Section~\ref{sec:duality} we opted for the largest possible
        solution: we use all segments. An advantage of this is that we can
        define bounded morphisms between (descriptive) frames as usual.
        On the other hand, as we saw in Section~\ref{subsec:Sahl-compl}
        this choice of descriptive frames forced us to add an extra
        construction to the proof of Sahlqvist canonicity.
        It would be interesting to see if there exists a middle ground
        between these competing interests.
  \item[Expressivity of the diamond-free fragment]
  Many developments of intuitionistic modal logic eschew the $\Diamond$ modality,
  proceeding with $\Box$ alone. Hence, there may be value in expressivity results, such
  as a Goldblatt-Thomason style theorem, for the logic without $\Diamond$. In particular the
  comparison of these results for the logics with and without $\Diamond$ would help us
  understand when $\Diamond$ is in fact necessary to capture frame conditions of interest.
  However the problem of the axiomatisation of the $\Diamond$-free fragment of Intuitionistic
  $\log{K}$~\cite{DasMar23} will not be resolved so easily, because there the question is
  the \emph{finite} axiomatisability of a logic, a considerably more difficult phenomenon to study.
\end{description}

%%%%%%%%%%%%%%%%%%%%%%%%%%%%%%%%%%%%%%%%%%%%%%%%%%%%%%%%%%%%%%%%%%%%%%%%%%%%%%%%
{
\footnotesize
\bibliographystyle{plainnat}
\bibliography{modal-int.bib}
}

\end{document}